\documentclass[dvipdfmx,a4paper,9pt]{article}

\usepackage{amsmath}
\usepackage{amsfonts}
\usepackage{amsthm}

\newtheorem{defn}{Definition}[section]
\newtheorem{lem}[defn]{Lemma}
\newtheorem{prop}[defn]{Proposition}
\newtheorem{thm}[defn]{Theorem}
\newtheorem{cor}[defn]{Corollary}

\newtheorem{rem}[defn]{Remark}
\newtheorem{ass}[defn]{Assumption}
\newtheorem{note}[defn]{Notation}
\newtheorem{ex}[defn]{Example}
\newtheorem*{claim}{Claim}
\newtheorem{step}{Step}

\newcommand{\A}{\mathcal A}
\newcommand{\B}{\mathcal B}
\newcommand{\G}{\mathcal G}
\newcommand{\su}{\mathfrak{su}(2)}

\def\ri{\rightarrow}

\def\Om{\Omega}

\newcommand{\n}{\mathbb N}
\newcommand{\R}{\mathbb R}
\newcommand{\z}{\mathbb Z}

\newcommand{\F}{\mathcal{F}}

\def\inner<#1>{\langle #1 \rangle}
\def\stab{\text{Stab}}
\def\coker{\text{Coker}}

\newcommand{\xri}{\xrightarrow}
\newcommand{\fr}{\mathfrak}

\def\pr{\text{pr}}
\def\ker{\text{Ker}}
\def\im{\text{Im}}

\def\hol{\text{Hol}}
\def\aut{\text{Aut}}
\def\map{\text{Map}}

\def\ome{\Omega}

\def\ad{\text{ad }}
\def\grad{\text{grad}}

\def\ind{\text{ind}}

\def\hom{\text{Hom}}
\def\End{\text{End}}

\def\ind{\text{ind}}

\begin{document}

\title{Instantons for 4-manifolds with periodic ends and an obstruction to embeddings of 3-manifolds}

\author{Masaki Taniguchi}

\maketitle
\begin{abstract}
We construct an obstruction to the existence of embeddings of a homology $3$-sphere into a homology $S^3\times S^1$ under some cohomological condition. The obstruction are defined as an element in the filtered version of the instanton Floer cohomology due to \cite{FS92}. We make use of the $\z$-fold covering space of homology $S^3\times S^1$ and the instantons on it.
\end{abstract}
\tableofcontents

\section{Introduction}
 There are two typical studies of gauge theory for 4-manifolds with periodic end by C.H.Taubes \cite{T87} and  J.Lin \cite{L16}. They gave a sufficient condition to exist a natural compactification of the instanton and the Seiberg-Witten moduli spaces for such non-compact 4-manifolds. The condition of Taubes is the non-existence of non-trivial $SU(2)$ flat connection of some segment of the end. The condition of Lin is the existence of a positive scalar curvature metric on the segment of the end. As an application of the existence of the compactification, Taubes showed the existence of uncountable family of exotic $\R^4$ and Lin constructed an obstruction to the existence of a positive scalar curvature metric.

 In this paper, we also give a similar sufficient condition for the instanton moduli spaces. The condition is an uniform bound on the $L^2$-norm of curvature of instantons. When we bound the $L^2$-norm of curvature, we use an  invariant which is a generalization of the Chern-Simons functional. Under this condition, we prove a compactness theorem (Theorem \ref{cptness}).

 As the main theorem of this paper, we construct an obstruction of the existence of embeddings of a homology $3$-sphere into a homology $S^3\times S^1$ with some cohomological condition (Theorem \ref{mainthm}). 
To formulate the obstruction, we need a variant of the instanton Floer cohomology. The variant is the filtered instanton Floer cohomology $HF^i_r$ whose filtration was essentially considered by R.Fintushel-R.Stern in \cite{FS92}. The obstruction is an element of the filtered instanton cohomology. We denote the element by $[\theta^r]$. The element $[\theta^r]\in HF^1_r$ is a filtered version of $[\theta]$ which was already defined by S.K.Donaldson \cite{Do02} and K.Fr\o yshov \cite{Fr02}. The class $[\theta]$ is defined by counting the gradient lines of the Chern-Simons functional which converge to the trivial flat connection. In order to show $[\theta^r]$ is actually an obstruction of embeddings, we count the number of the end of the $1$-dimensional instanton moduli space for 4-manifold which has both of cylindrical and periodic end. For the couting, we use the compactness theorem (Theorem \ref{cptness}).

This paper is organized as follows. In Section \ref{main}, we give a precise formulation of our main theorem (Theorem \ref{mainthm}). In Section \ref{moduli theory}, we prepare several notations and constructions which are used in the rest of this paper.
In particular, we introduce the filtered instanton Floer homology $HF^r_i$ and the obstruction class $[\theta^r]$. We also review Fredholm, moduli theory for 4-manifolds with periodic end. In Section \ref{excs}, we generalize the Chern-Simons functional and introduce the invariants $Q^{i}_X$. In Section 5, we prove the compactness theorem(Theorem \ref{cptness}). We use $Q^{i}_X$ to control the $L^2$-norm of curvature. In Section 6, we deal with technical arguments about the transversality and the orientation for the instanton moduli spaces for 4-manifolds with periodic end. In Section 7, we prove Theorem \ref{mainthm}.

 {\bf Acknowledgements}:
The author is grateful to Mikio Furuta for his suggestions. The author would like to thank Hokuto Konno for useful conversations.

\section{Main theorem}\label{main}
Let $X$ be a homology $S^3\times S^1$, i.e.\ , $X$ is a closed 4-manifold equipped with an isomorphism $\phi:H_*(X,\z)\ri H_*(S^3\times S^1,\z)$ in this paper. Then $X$ has an orientation induced by the standard orientation of $S^3\times S^1$ and $\phi$.  Let $Y$ be an oriented homology $S^3$. We construct an obstruction of embeddings $f$ of $Y$ into $X$ satisfying  $f_*[Y]=1\in H_3(X,\z)$ as an element in the filtered instanton Floer cohomology. We use information of the compactness of the instanton moduli spaces for periodic-end 4-manifold in a crucial step of our construction.  In order to formulate our main theorem, we need to prepare several notations.

For any manifold $Z$, we denote by $P_Z$ the product $SU(2)$ bundle. The product connection on $P_Z$ is written by $\theta$.
\[
\A(Z):= \{SU(2)\text{-connections on $P_Z$}\},
\]
\[
\A^{\text{flat}}(Z):=\{SU(2)\text{-flat connections on $P_Z$}\}\subset \A(Y),
\]
\[
\widetilde{\B}(Z):= \A(Z) /\map_0(Z,SU(2)),
\]
\[
\widetilde{R}(Z):=\A^{\text{flat}} /\map_0(Z,SU(2))\subset \widetilde{\B}(Z),
\]
and
\[
R(Z):= \A^{\text{flat}}(Z)/\map(Z,SU(2)),
\]
where $\map_0(Z,SU(2))$ is a set of smooth functions with mapping degree $0$. When $Z$ is equal to $Y$, the Chern-Simons functional $cs_Y:\A(Y)\ri \R$ is defined by 
\[
cs_Y(a):=\frac{1}{8\pi^2}\int_Y Tr(a\wedge da +\frac{2}{3}a\wedge a\wedge a).
\]
It is known that $cs_Y$ decends to a map $\widetilde{\B}(Y)\ri \R$, which we denote by the same notation $cs_Y$.
\begin{note}\label{defl}
We denote the number of elements in $R(Y)$ by $l_Y$. If $R(Y)$ is not a finite set, we set $l_Y=\infty$.
\end{note}
We will use the following assumption on $Y$ in our main theorem(Theorem \ref{mainthm}).
\begin{ass}\label{imp}
All $SU(2)$ flat connections on $Y$ are non-degenerate, i.e. the first cohomology group of the next twisted de Rham complex:
\[
0 \ri \Om^0(Y)\otimes \su \xri{d_a} \Om^1(Y)\otimes \su \xri{d_a} \Om^2(Y)\otimes \su \xri{d_a}\Om^3\otimes \su \ri 0
\]
vanishes for $[a] \in R(Y)$.
\end{ass}
\begin{ex}
All flat connections on the Brieskorn homology $3$-sphere $\Sigma(p,q,r)$ are non-degenerate.(\cite{FS90})
\end{ex}
Under Assumption \ref{imp}, $l_Y$ is finite (\cite{T90}).

In this paper without the use of Assumption \ref{imp}, we will introduce the following invariants:
\begin{itemize}
\item $HF^i_r(Y)$ for $Y$ and $r \in (\R \setminus cs_Y(\widetilde{R}(Y)))\cup \{\infty\}$ in Definition \ref{defofHFr*}  satisfying $HF^i_\infty(Y)=HF^i(Y)$,
\item $[\theta^r] \in HF^1_r(Y)$ for $Y$ and $r \in (\R \setminus cs_Y(\widetilde{R}(Y)))\cup \{\infty\}$ in Definition \ref{defiofthetar} satisfying $[\theta^\infty]=[\theta] \in HF^1(Y)$, and
\item $Q^i_X \in \R_{\geq 0} \cup \{\infty\} $ for $i \in \n$ and $X$ in Definition \ref{defofQiX} (When $X$ is homotopy equivalent to $S^3\times S^1$, $Q^i_X=\infty$ for all $i \in \n$).

\end{itemize}
Our main theorem is:
\begin{thm}\label{mainthm}Under Assumption $\ref{imp}$, if there exists an embedding $f$ of $Y$ into $X$ with $f_*[Y]=1\in H_3(X,\z)$  then $[\theta^r]$ vanishes for any $r\in [0,\min \{Q^{2l_Y+3}_X, 1\}] \cap (\R \setminus cs_Y(\widetilde{R}(Y))\cup \{\infty\})$ 
\end{thm}
In particular, if there exists an element 
\[
r\in [0,\min\{Q^{2l_Y+3}_X,1\}] \cap (\R \setminus cs_Y(\widetilde{R}(Y))\cup \{\infty\})
\]
satisfying $0\neq [\theta^r]$, Theorem \ref{mainthm} implies that there is no embedding from $Y$ to $X$ with $f_*[Y]=1 \in H_3(X,\z)$.
\begin{ex}\label{exbri}
Let $X$ be a closed $4$-manifold which is homotopy equivalent to $S^3\times S^1$.
There is no embedding $f$ of $\Sigma(2,3,6k-1)$ into $X$ satisfying $f_*[\Sigma(2,3,6k-1)] =1\in H_3(X,\z)$ for a positive integer $k$ satisfying $1\leq k \leq 12$.
 \end{ex}
The proof of Example \ref{exbri} is given in the end of Subsection \ref{filter}.

\section{Preliminaries}\label{moduli theory}
In this section, we review the (filtered) instanton Floer theory and moduli theory on the periodic end 4-manifolds.\subsection{Holonomy perturbation(1)}\label{hol}
In this subsection we review classes of perturbations which were considered in \cite{Fl88}, \cite{BD95} to define the instanton Floer homology.

Let $Y$ and $P_Y$ be as in Section \ref{main}. We fix a Riemannian metric $g_Y$ on $Y$. We define the set of embeddings from solid tori to $Y$ by
\[
\F_d:= \left\{ (f_i)_{1\leq i \leq d} :S^1\times D^2\ri Y \middle | f_i: \text{orientation preserving embedding} \ \right\}.
\]
Fix a two form $d\mathcal{S}$ on $D^2$ supported in the interior of $D^2$ with $\int_{D^2}d\mathcal{S}=1$. We denote by $C^l(SU(2)^d,\R)_{\ad}$ the adjoint invariant $C^l$-class functions from $SU(2)^d$ to $\R$ and define 
\[
\prod(Y):= \bigcup_{d \in \n}\F_d\times C^l(SU(2)^d,\R)_{\ad}.
\]

We use the following notation,
\[
\widetilde{\B}^*(Y):=\left\{[a]\in \widetilde{\B}(Y)\middle | \text{$a$ is an irreducible connection} \right\} ,
\]
where $\widetilde{\B}(Y)$ is defined in Section \ref{main}.
For $\pi = (f,h)\in \prod(Y)$, the perturbed Chern-Simons functional $cs_{Y,\pi}:\widetilde{\B}^*(Y) \ri \R$  is defined by
\[
cs_{Y,\pi}(a)= cs_Y(a)+ \int_{x \in D^2} h(\hol(a)_{f_1(-,x)},\dots ,\hol(a)_{f_d(-,x)})d\mathcal{S},
\]
where $\hol(a)_{f_i(-,x)}$ is the holonomy around the loop $t \mapsto f_i(t,x)$ for each $i \in \{1,\dots,d\}$. If we identify $\su$ with its dual by the Killing form, the derivative of $h_i=pr_i^*h$ is a Lie algebra valued 1-form over $SU(2)$ for $h \in C^l(SU(2)^d,\R)_{\ad}$. Using the value of holonomy for the loops $\{f_i(x,t)| t\in S^1\}$, we obtain a section $\hol_{f_i(t,x)}(a)$ of the bundle \aut $P$ over $\im f_i$. Sending the section $\hol_{f_i(t,x)}(a)$ by the bundle map induced by $h_i':\aut P\ri \ad P$, we obtain a section $h_i'(\hol_{f_i(t,x)}(a))$ of $\ad P$ over $\im f_i$. 
We now describe the gradient-line equation of $cs_{Y,\pi}$ with respect to $L^2$ metric :
\begin{align}\label{grad}
 \frac{\partial}{\partial t} a_t=\grad_a\ cs_{Y,\pi} = *_{g_Y}(F(a_t)+\sum_{1\leq i \leq d} h'_i(\hol(a_t)_{f_i(t,x)})(f_i)_*{\pr}_2^*d\mathcal{S}),
\end{align}
where $\pr_2$ is the second projection $\pr_2:S^1\times D^2 \ri D^2$ and $*_{g_Y}$ is the Hodge star operator. We denote $\pr_2^*d\mathcal{S}$ by $\eta$.
We set
\[
\widetilde{R}(Y)_\pi:= \left\{a \in \widetilde{\B}(Y) \middle |F(a)+\sum_{1\leq i \leq d} h'_i(\hol(a)_{f_i(t,x)})(f_i)_*\eta=0 \right\},
 \]
 and 
 \[
 \widetilde{R}^*(Y)_\pi:= \widetilde{R}(Y)_\pi \cap \widetilde{\B}^*(Y).
 \]
The solutions of \eqref{grad} correspond to connections $A$ over $Y\times \R$ which satisfy an equation:
\begin{align}\label{pASD}
F^+(A)+ \pi(A)^+=0,
\end{align}
where 
\begin{itemize}
\item The two form $\pi(A)$ is given by 
\[
\sum_{1\leq i \leq d} h'_i(\hol(A)_{\tilde{f}_i(t,x,s)}){(\tilde{f}_i)}_* (\pr_1^* \eta).
\]
\item The map $\pr_1$ is a projection map from $(S^1\times D^2) \times \R$ to $S^1\times D^2$.
\item The notation $+$ is the self-dual component with respect to the product metric on $Y\times \R$.
\item The map $\tilde{f}_i: S^1\times D^2\times \R \ri Y\times \R$ is $f_i\times id$. 
\end{itemize}
We also use several classes of the perturbations.
\begin{defn}\label{flatpres}\upshape
A class of perturbation $\prod(Y)^{\text{flat}}$ is defined by a subset of $\prod(Y)$ with the conditions:
\begin{itemize}
\item $cs_Y $ coincides  with $cs_{Y,\pi}$ on a small neighborhood of critical points of $cs_Y$ 
\item $\widetilde{R}(Y)=\widetilde{R}(Y)_\pi$,
 \end{itemize}
for all element in $\prod(Y)^{\text{flat}}$.
\end{defn}
If the cohomology groups defined by the complex (12) in \cite{SaWe08} satisfies $H^i_{\pi,a}=0$ for all $[a] \in \widetilde{R}(Y)_\pi \setminus \{ [\theta]\}$ for $\pi$, we call $\pi$ {\it non-degenerate perturbation}. 
If $\pi$ satisfies the following conditions, we call $\pi$ {\it regular perturbation}. 
\begin{itemize}
\item The linearization of \eqref{pASD}
\[
d^+_A+d\pi^+_A : \Om^1(Y\times \R)\otimes \su_{L^2_q}\ri \Om^+(Y\times \R)\otimes \su_{L^2_q}
\]
 is surjective for $[A] \in M(a,b)_\pi$ and  all irreducible critical point $a,b$ of $cs_{Y,\pi}$.
 \item The linearization of \eqref{pASD}
 \[
d^+_A+d\pi^+_A : \Om^1(Y\times \R)\otimes \su_{L^2_{q,\delta}}\ri \Om^+(Y\times \R)\otimes \su_{L^2_{q,\delta}}
\]
is surjective for $[A] \in M(a,\theta)_\pi$ and  all irreducible critical point $a$ of $cs_{Y,\pi}$.
 \end{itemize}
 Here the spaces  $M(a,b)_\pi$ and $M(a,\theta)_\pi$ are given in \eqref{*} in Subsection \ref{filter}, $L^2_q$ is the Sobolev norm and $L^2_{q,\delta}$ is the weighted Sobolev norm which is same one as in Subsection $3.3.1$ in \cite{Do02}.

\subsection{Filtered instanton Floer (co)homology}\label{filter}
In this subsection, we give the definition of the filtration of the instanton Floer (co)homology by using the technique in \cite{FS92}. First, we give the definition of usual instanton Floer homology.

Let $Y$ be a homology $S^3$ and fix a Riemannian metric $g_Y$ on $Y$. Fix a non-degenerate regular perturbation $\pi \in \prod(Y)$.
Roughly speaking, the instanton Floer homology is inifinite dimensional Morse homology with respect to 
\[
cs_{Y,\pi} :\widetilde{\B}^*(Y) \ri \R.
\]

Floer defined $\ind: \widetilde{R}^*(Y)_\pi \ri \z$, called the Floer index. 
The (co)chains of the instanton Floer homology are defined by
\[
CF_i:= \z \left\{ [a] \in \widetilde{R}^*(Y)_\pi \middle | \ind(a)=i \right\}(CF^i:= \hom(CF_i,\z)).
\]
The boundary maps $\partial :CF_i \ri CF_{i-1}(\delta:CF^i\ri CF^{i+1})$ are given by
\[
\partial ([a]) := \sum_{b \in \widetilde{R}^*(Y)_\pi \text{ with } \ind(b)=i-1}\# (M(a,b)/\R) [b]\ (\delta:=\partial^*),
\] 
where $M(a,b)$ is the space of trajectories of $cs_{Y,\pi} $ from $a$ to $b$. We now write the explicit definition of $M(a,b)$. Fix a positive integer $q\geq3$. Let $A_{a,b}$ be an $SU(2)$ connection on $Y \times \R$ satisfying $A_{a,b}|_{Y\times (-\infty,1]}=p^*a$ and $A_{a,b}|_{Y\times [1,\infty)}=p^*b$ where $p$ is projection $Y\times \R \ri Y$.

\begin{align}\label{*}
M(a,b)_\pi:=\left\{A_{a,b}+c  \middle | c \in \Omega^1(Y\times \R)\otimes \su_{L^2_q}\text{ with } \eqref{pASD} \right\}/ \G(a,b),
\end{align}
where $\G(a,b)$ is given by 
\[
\G(a,b):=\left\{ g \in \aut(P_{Y\times \R})\subset {\End(\mathbb{C}^2)    }_{L^2_{q+1,\text{loc}}} \middle | \nabla_{A_{a,b}}(g) \in L^2_q \right\}.
\]
The action of $\G(a,b)$ on $\left\{A_{a,b}+c  \middle | c \in \Omega^1(Y\times \R)\otimes \su_{L^2_q}\text{ with }\eqref{pASD} \right\}$ given by the pull-backs of connections. The space $\R$ has an  action on $M(a,b)$ by the translation. Floer show that $M(a,b)/\R$ has structure of a compact oriented 0-manifold whose orientation is induced by the orientation of some determinant line bundles and $\partial^2=0$ holds.
The instanton Floer (co)homology $HF_*(Y)(HF^*(Y))$ is defined by 
\[
HF_*(Y):= \ker \partial  / \im \partial \ (HF^*(Y):= \ker \delta  / \im \delta).
\]

Second, we introduce the filtration in the instanton Floer homology. This filtration is essentially considered by Fintushel-Stern in \cite{FS92}. We follow Fintushel-Stern and use the class of perturbations which they call $\epsilon$-perturbation defined in Section $3$ of \cite{FS92}. They constructed $\z$-graded Floer homology whose chains are generated by the critical points of $cs_{Y,\pi}$ with $cs_Y(a)\in (m,m+1)$. We now consider Floer homology whose chains generated by the critical points of $cs_{Y,\pi}$ with $cs_Y(a)\in (-\infty,r)$.

Let $\widetilde{R}(Y) $ be as in Section \ref{main} and $\Lambda_Y$ be $\R \setminus \im \ cs_Y|_{\widetilde{R}(Y)}$. For $r \in \Lambda_Y$, we define the filtered instanton homology $HF^r_*(Y) (HF^*_r(Y))$ by using $\epsilon$-perturbation.  For $r \in  \Lambda_Y$, we set $\epsilon= \inf_{a \in \widetilde{R}(Y)} |cs_Y(a)-r|$ and choose such a $\epsilon$-perturbation $\pi$.  
\begin{defn}[Filtered version of the instanton Floer homology]\label{defofHFr*}\upshape
The chains of the filtered instanton Floer (co)homology are defined by
\[
CF^r_i:= \z \left\{ [a] \in \widetilde{R}^*(Y)_\pi \middle | \ind(a)=i,\ cs_{Y,\pi}(a)<r \right\}\ (CF^i_r:= \hom(CF_r^i,\z)).
\]
The boundary maps $\partial^r :CF_i^r \ri CF_{i-1}^r$(resp. $\delta^r:CF^i_r\ri CF^{i+1}_r $) are given by the restriction of $\partial$ to $CF_i^r$(resp. $\delta^r:=(\partial^r)^*$). This maps are well-defined and $(\partial^r)^2=0$ holds as in Section $4$ of \cite{FS92}.
The filtered instanton Floer (co)homology $HF^r_*(Y)$(resp. $HF_r^*(Y))$ is defined by 
\begin{align*}
HF^r_*(Y):= \ker \partial^r  / \im \partial^r \ (\text{resp. }HF^*_r(Y):= \ker \delta^r  / \im \delta^r).
\end{align*}
\end{defn}
We can also show $HF^r_i(Y)$ and $HF^i_r(Y)$ are independent of the choices of the perturbation and the metric by  similar discussion in \cite{FS92}.
For $r \in \Lambda_Y$, we now introduce obstruction classes in $HF^1_r(Y)$. These invariants are  generalizations of  $[\theta] \in HF^1(Y)$ considered in Subsection $7.1$ of \cite{Do02} and Subsection $2.1$ of \cite{Fr02}. 
\begin{defn}[Obstruction class]\label{defiofthetar}\upshape
For $r \in \Lambda_Y$, we set homomorphism 
\[
\theta^r :CF^r_1\ri \z
\]
 by
\begin{align}
\theta^r(a):= \# (M(a,\theta)_\pi/\R).
\end{align}
As in \cite{Do02} and \cite{Fr02}, we use the weighted norm on $M(a,\theta)_\pi$ to use Fredholm theory. From the same discussion for the proof of $(\delta^r)^2=0$, we can show $\delta^r (\theta^r)=0$. Therefore it defines the class $[\theta^r] \in HF^1_r(Y)$. We call the class $[\theta^r]$ {\it obstruction class}.
\end{defn}
The class $[\theta^r]$ does not depend on the small perturbation and the metric. The proof is similar to the proof for original one $[\theta]$. Now we give the proof of Example \ref{exbri}.
\\
\begin{proof}
Because $X$ is homotopy equivalent to $S^3\times S^1$, $Q^i_X=\infty$ for $i \in \n$. If the element $[\theta^1] \in HF^1_r(-\Sigma(2,3,6k-1))$ does not vanish for $r=1$, we can apply Theorem \ref{mainthm}. Fr\o yshov showed $0 \neq [\theta] \in HF^1(-\Sigma(2,3,6k-1))$ by using the property of h-invariant of Proposition $4$ in \cite{Fr02}. Then we get nonzero homomorphism $\theta:CF^r_1(Y)\ri \z$ for $r=\infty$. ( If $r=\infty$, $HF^1_r$ is the usual instanton Floer cohomology by the definition.) 
By using calculation about the value of the Chern-Simons functional of Section 7 in \cite{FS92}, 
 we can see $\theta^1:CF_1^r\ri \z$ is nonzero for $r=1$ and $CF_i^r$ is zero for $r=1$ and $i\in 2\z$. This implies 
 \[
 0 \neq [\theta^1] \in HF^1_r(-\Sigma(2,3,6k-1)) \text{ for } r=1.
 \]
 \qed
\end{proof}

\subsection{Fredholm theory and moduli theory on 4-manifolds with periodic end}\label{Fred}
 In \cite{T87}, Taubes constructed the Fredholm theory of some class of elliptic operators on 4-manifolds with periodic ends. He also extend moduli theory of $SU(2)$ gauge theory on such non-compact 4-manifolds.
In this subsection, we review Fredholm theory of a certain elliptic operator on $4$-manifolds with the periodic ends  as in \cite{T87} and define the Fredholm index of the class of operators, which gives the formal dimension of a suitable instanton moduli space on such non-compact 4-manifolds. First we formulate the 4-manifolds with periodic ends.

Let $Y$ be an oriented homology $S^3$ as in Section \ref{main}. 
Let $W_0$ be an oriented homology cobordism from $Y$ to $-Y$. We get a compact oriented 4-manifold $X$ by pasting $W_0$ with itself along $Y$ and $-Y$. We give several notations in our argument.
\begin{itemize}
\item The manifold $W_i$ is a copy of $W_0$ for $i \in \z$
\item We denote $\partial(W_i)$ by $Y^i_+\cup Y^i_-$ where $Y^i_+$(resp. $Y^i_-$) is equal to $Y$(resp. $-Y$) as oriented manifolds.
\item For  $(m,n)\in (\z\cup \{ -\infty\} ) \times (\z\cup \{\infty\})$ with $m<n$, we set
\[
\displaystyle W[m,n]:=\coprod_{m\leq i \leq n} W_i  / \{Y^j_- \sim Y^{j+1}_+\ j \in \{m,\cdots ,n\}\}.
\]
\end{itemize}

We denote by $W$ the following non-compact 4-manifold
\[
W:= Y\times (-\infty,0] \cup W[0,\infty]/\{\partial (Y\times (-\infty,0]) \sim Y^{0}_+\}.
\] 
For a fixed Riemannian metric $g_Y$ on $Y$, we choose a Riemannian metric $g_W$ on $W$ which satisfies
\begin{itemize}
\item $g_W|_{Y\times (-\infty,-1]}=g_Y\times g^{\text{stan}}_\R$.
\item The restriction $g_W|_{W[0,\infty]}$ is a periodic metric.
\end{itemize}
There is a natural orientation on $W[0,\infty]$ and $W$ induced by the orientations of $W_0$.  The infinite cyclic covering space of $X$ can be written by 
\[
\widetilde{X} \cong W[-\infty,\infty].
\]
Let $T$ be the deck transformation of $\widetilde{X}$ which maps each $W_i$ to $W_{i+1}$. By restriction, $T$ has an action on $W[0,\infty]$. We use the following smooth functions $\tau$ and $\tau'$ on $W$
\[
\tau,\ \tau': W \ri \R\,
\]
satisfying 
\begin{itemize}
\item $\tau (T|_{W[0,\infty]}(x))= \tau(x)+1$, $\tau'(T|_{W[0,\infty]}(x))=\tau'(x)+1$ for $x \in W[0,\infty]$.
\item $\tau|_{Y\times (-\infty,-2]}=0$, $\tau'(y,t)= -t$ for $(y,t) \in Y\times (-\infty,-2]$.
\end{itemize}
 By the restriction of $\tau$, we have a function on $W[0,\infty]$ which we denote by same notation $\tau$.

In this subsection, we review the setting of the configuration space of fields on $W$ and define the Fredholm index of a kind of operator on $W$. We fix $\pi \in \prod(Y)$ in Subsection $\ref{hol}$ and assume that $\pi$ is a non-degenerate perturbation. Let $P_W$ be the product $SU(2)$ bundle. 
\begin{defn}\label{confset}\upshape
For each element $[a] \in \widetilde{R}(Y)$, we fix an $SU(2)$ connection $A_a $ on $P_W$ which satisfying $A_{a}|_{Y\times (-\infty,-1]} =\pr_1^*a$ and $A_{a}|_{W[0,\infty]}=\theta$. If $a$ is an irreducible (resp. reducible) connection, we define the space of connections on $P_W$ by
\[
\A^W(a)_\delta := \left\{ A_{a}+c \ \middle| c \in \Omega^1(W)\otimes \su_{L^2_{q,\delta}}  \right\},
 \]
 \[
 (\text{ resp.}\  \A^W(a)_{(\delta,\delta)} := \left\{ A_{a}+c \ \middle|  c \in \Omega^1(W)\otimes \su_{L^2_{q,(\delta,\delta)}}\right\}\ )
 \]
where $\Omega^1(W)\otimes \su_{L^2_{q,\delta}}$(resp. $\Omega^1(W)\otimes \su_{L^2_{q,(\delta,\delta)}}$) is the completion of $\Omega^1(W)\otimes \su $ with $L^2_{q,\delta}$-norm(resp. $L^2_{q,(\delta,\delta)}$-norm), $q $ is a natural number greater than $3$, and $\delta$ is a positive real number. For $f \in \Omega^i(W)\otimes \su$ with compact support, we define $L^2_{q,\delta}$-norm(resp. $L^2_{q,(\delta,\delta)}$-norm) by 
\[
||f||^2_{L^2_{q,\delta}}:= \sum_{0\leq j \leq q} \int_W e^{\delta \tau}  \left|\nabla_{\theta}^j f \right|^2 d\text{vol} ,
\]
\[
(\text{resp. } ||f||^2_{L^2_{q,(\delta,\delta)}}:= \sum_{0\leq j \leq q} \int_W e^{\delta \tau'}  \left|\nabla_{\theta}^j f \right|^2 d\text{vol}\  )
\]
where $\nabla_{\theta}$ is covariant derivertive with respect to the product connection $\theta$. We use a periodic metric $||-||$ on the bundle. Its completion is denoted by $\Omega^i(W)\otimes \su_{L^2_{q,\delta}}$.
We define the gauge group 
\[
\G^W(a)_{\delta}:= \left\{ g \in \aut(P_W)_{L^2_{q+1,\text{loc}}} \middle| \nabla_{A_{a}}(g) \in L^2_{q,\delta}
  \right\}
\]
\[
( \text{ resp. }\G^W(a)_{(\delta,\delta)}:= \left\{ g \in \aut(P_W)_{L^2_{q+1,\text{loc}}} \middle| \nabla_{A_{a}}(g) \in L^2_{q,(\delta,\delta)}  
  \right\}),
\]
which has the action on $\A^W(a)_\delta$ induced by the pull-backs of connections.
The space $\G^W(a)_{\delta}$(resp. $\G^W(a)_{(\delta,\delta)}$) has structure of Banach Lie group and the action of $\G^W(a)_{\delta}$ on $\A^W(a)_\delta$(resp. $\G^W(a)_{(\delta,\delta)}$ on $\A^W(a)_{(\delta,\delta)}$) is smooth.
{\it The configuration space for} $W$ is defined by 
\[
\B^W(a)_\delta := \A^W(a)_\delta /\G^W(a)_{\delta} (\text{resp. } B^W(a)_{(\delta,\delta)} := \A^W(a)_{(\delta,\delta)} /\G^W(a)_{(\delta,\delta)}).
\]
Let $s$ be a smooth function from $W$ to $[0,1]$ with 
\[
s|_{Y\times (-\infty -2]}=1,\ s|_{Y\times [-1,0] \cup_Y W[0,\infty]}=0.
\]
We define the instanton moduli space for $W$ by 
\begin{align}
M^W(a)_{\pi,\delta}:=\{[A] \in \B^W(a)_\delta |  \mathcal{F}_{\pi}(A)=0 \}
\end{align} 
where $\mathcal{F}_{\pi}$ is the perturbed ASD-map 
\[
\mathcal{F}_{\pi}(A):=F^+(A)+s\pi(A).
\]
\end{defn}
For each $A \in \A^W(a)_\delta$, we have the bounded linear operator:
\begin{align}\label{elli}
d^{*_{L^2_\delta}}_A+d\mathcal{F}_A: \Om^1(W)\otimes \su_{L^2_{q,\delta}}  \ri (\Om^0(W)  \oplus \Om^+(W))\otimes \su_{L^2_{q-1,\delta}}
\end{align}
\begin{align}\label{thetacase}
(\text{ resp. }d^{*_{L^2_{(\delta,\delta)}}}_A+d\mathcal{F}_A: \Om^1(W)\otimes \su_{L^2_{q,(\delta,\delta)}}  \ri (\Om^0(W)  \oplus \Om^+(W))\otimes \su_{L^2_{q-1,(\delta,\delta)}}).
\end{align}

Taubes gave a criterion for the operator $d^{*_{L^2_\delta}}_A+d\mathcal{F}_A=d^{*_{L^2_\delta}}_A+ d^+_A+sd\pi^+_A$ in $\eqref{elli}$(resp. \eqref{thetacase}) to be Fredholm in Theorem $3.1$ of \cite{T87}.

\begin{thm}[Taubes,\cite{T87}]\label{fred}
There exists a descrete set $D$ in $\R$ with no accumulation points such that $\eqref{elli}$$($resp. \eqref{thetacase}$)$ is Fredholm for each $\delta$ in $\R \setminus D$.
\end{thm}
The discrete set $D$ is defined by 
\[
D:=\{\delta \in \R |   \text{ the cohomology groups } H^i_z\text{ are acyclic for all $z$ with } |z|=e^{\frac{\delta}{2}} \}.
\]
The cohomology groups $H^i_z$ are given by the complex:
\begin{align}\label{cpx}
0\ri \Om^0(X) \otimes \su \xri{d_{\theta,z}} \Om^1(X)\otimes \su \xri{d^+_{\theta,z}} \Om^+(X)\otimes \su \ri 0,
\end{align}
where
 \[
 d_{\theta,z}:\Om^0(X) \otimes \su \xri{} \Om^1(X)\otimes \su 
 \]and
 \[
 d^+_{\theta.z} :\Om^1(X)\otimes \su \ri \Om^+(X)\otimes \su
 \]
are given by 
 \[
 d_{\theta,z}(f)=z^\tau d_{p^*\theta} (z^{-\tau} (p^*f)),\  d^+_{\theta,z}(f)=z^\tau d^+_{p^*\theta} (z^{-\tau}(p^*f)),
 \]
 where $p$ is the covering map $\widetilde{X}\ri X$.
 (We fix a branch of ln $z$ to define $z^{\tau} = e^{\tau ln z}$.) In above definition,  $d_{\theta,z}(f)$ and $d^+_{\theta,z}(f)$ are sections of $p^*P_{X}$, however these are invariant under the deck transformation, we regard $d_{\theta,z}(f)$ and $d^+_{\theta,z}(f)$
as sections on $P_X$.

 The operators $d_{\theta,z}$ and $d^+_{\theta.z}$ in \eqref{cpx} depend on the metric on $X$ and $\tau$ however the cohomology groups $H_i^z$ are independent of the choice of them.
We now introduce the formal dimension of the instanton moduli spaces:
\begin{defn}\upshape
Suppose $a$ is an irreducible critical point of $cs_{Y,\pi}$.
From Theorem \ref{fred}, there exists $\delta_0>0$  such that $\eqref{elli}$ is Fredholm for any $\delta \in (0,\delta_0)$ and $A \in \A^W(a)_{\delta}$. We define {\it the formal dimension} $\ind_W(a)$ of the instanton moduli spaces for $W$ by the Fredholm index of \eqref{elli}. For the case of $a=\theta$, we also set $\ind_W(a)$ as the Fredholm index of \eqref{thetacase}.
\end{defn}

The formal dimension $\ind_W(a)$ is calculated by using the following proposition.
\begin{prop}\label{calfred}
Suppose $a$ is an irreducible critical point of $cs_{Y,\pi}$.
The formal dimension $\ind_W(a)$ is equal to the Floer index $\ind(a)$ of $a$. If $a$ is equal to $\theta$, $\ind_W(a)=\ind (a)=-3$.
\end{prop}
\begin{proof}
First we take a compact oriented 4-manifold $Z$ with $\partial Z=-Y$. It is easy to show there is an isomorphism $H^*(W[0,\infty]) \cong H^*(S^3)$. We define $Z^+:= Z\cup_Y W[0,\infty]$ and fix a periodic Riemannian metric $g_{Z^+}$ satisfying $g_{Z^+}|_{W[0,\infty]}=g_W$. In Proposition $5.1$ of \cite{T87}, Taubes computed the Fredholm index of $d^+_\theta+ d^{*_{L^2_\delta}}_\theta $ as a operator on $Z^+$ in the situation that $H_*(W[0,\infty],\z)\cong H_*(S^3,\z)$(The proof is given by using the admissibility of each segment $W_0$, however Taubes just use the condition $H_*(W[0,\infty],\z)\cong H_*(S^3,\z)$.):
\[
\ind (d^+_\theta+ d^*_\theta) =-3(1-b_1(Z)+b^+(Z))
\]
for a small $\delta$. Fix an $SU(2)$-connection $A_{a,\theta}$ on $W$ with $A|_{Y\times (-\infty,-1]}=a$, $A|_{W[0,\infty]}=\theta$ and an $SU(2)$-connection $B_{a}$ on $Z\cup_Y Y\times [0,\infty)$ 
  By the similar discussion about gluing of the operators on cylindrical end in Proposition $3.9$ of \cite{Do02}, we have 
\[
\ind (d_{B_{a}}^*+d_{B_{a}}^+)+ \ind (d_{A_{a,\theta}}^*+d_{A_{a,\theta}}^+)= \ind (d_{\theta}^*+d_{\theta}^+).
\]
Donaldson show that $\ind (d_{B_{a}}^*+d_{B_{a}}^+)$ is equal to $-\ind(a)-3(1-b_1(Z)+b^+(Z))$ in Proposition $3.17$ of \cite{Do02}.
The second statement is similar to the first one.
\qed
\end{proof}

\section{Chern-Simons functional for homology $S^3\times S^1$}\label{excs}
For a pair $(X,\phi)$ consisting of an oriented 4-manifold and non-zero element $0\neq \phi \in H^1(X,\z)$, we  generalize the Chern-Simons functional to a functional $cs_{(X,\phi)}$ on the flat connections on $X$. We define the invariants $Q^i_X \in \R_{\geq0} \cup\{\infty\}$ for $i\in \n$ by using the value of $cs_{(X,\phi)}$. In our construction, $cs_{(X,\phi)}$ cannot be extended to a functional for arbitrary $SU(2)$ connections on $X$.

Let $X$ be an oriented closed 4-manifolds equipped with $0 \neq \phi \in H^1(X,\z)$ and $p:\widetilde{X}^{\phi}\ri X$ be the $\z$-hold covering of $X$ corresponding to $\phi \in H^1(X,\z)\cong [X,B\z]$. Recall that the bundle $P_{X} $ and the set $\widetilde{R}(X)$ as in Section \ref{main}. Let $f$ be a smooth map representing the class $\phi \in H^1(X,\z)\cong [X,S^1]$, and $\tilde{f}$ is a lift of $f$.

\begin{defn}\label{csX}\upshape [Chern-Simons functional for a homology $S^3\times S^1$]
We define the Chern-Simons functional for $X$ as the following map
\[
cs_{(X,\phi)}:\widetilde{R}(X)\times \widetilde{R}(X) \ri \R,
\]
\[
cs_{(X,\phi)}([a],[b]):= \frac{1}{8\pi^2}\int_{\widetilde{X}^\phi} Tr(F(A_{a,b})\wedge F(A_{a,b})),
\]
where $a,b$ are flat connections on $P_X$ and $A_{a,b}$ is an $SU(2)$-connection on $P_{\widetilde{X}^\phi}:=\widetilde{X}^\phi \times SU(2)$ which satisfies $A_{a,b}|_{\tilde{f}^{-1}(-\infty,-r]}=p^*a$ and $A_{a,b}|_{\tilde{f}^{-1}([r,\infty))}=p^*b$ for some $r>0$.
\end{defn}

We have an alternative description of $cs_{(X,\phi)}([a],[b])$ when a closed oriented $3$-manifold $Y$ is given as a sub-manifold of $X$ satisfying $i_*[Y]=\text{PD}(\phi) \in H_3(X,\z)$, where $i$ is the inclusion $Y\ri X$. Such $Y$ is given as an inverse image of a regular value of $f$. We can take $Y$ to be connected, and we assume this. We denote by $W_0$ the cobordism from $Y$ to itself obtained from cutting $X$ open along $Y$. Since $Y$ is connected and $\phi \neq 0$, $W_0$ is also connected. Then we choose the idenitification of $\widetilde{X}^\phi$ and $\dots \cup_Y W_0 \cup_Y W_1 \cup_Y \dots $. We have the following formula.
\begin{lem}\label{sumformula}
\[
cs_{(X,\phi)}([a],[b])=cs_Y([i^*a])-cs_Y([i^*b])
\]

\end{lem}
\begin{proof}
Let $A_{a,b}$ the $SU(2)$-connection on $P_{\widetilde{X}^\phi}$ in Definition \ref{csX}.
Take a natural number $N$ large enough to satisfy 
\[
\tilde{f}^{-1}([-r,r]) \subset W[-N,N],
\]
for which we have
\begin{align*}\label{cscal}
\int_{\widetilde{X}^\phi} Tr(F(A_{a,b})\wedge F(A_{a,b}))= \int_{W[-N,N]}Tr(F(A_{a,b})\wedge F(A_{a,b})).
\end{align*}
\[
=cs_Y(i^*_+A_{a,b})-cs_Y(i^*_-A_{a,b}).
\]
Here  $i_+$(resp. $i_-$) is inclusion from $Y$(resp. $-Y$) to $X$, and we use the Stokes theorem.
\qed
\end{proof}

When $X$ is equal to $Y\times S^1$, this map $cs_{(X,\phi)}:\widetilde{R}(X)\times \widetilde{R}(X) \ri \R$ essentially coincides with the restricton of Chern-Simons functional $cs_Y$ on $Y$ by the following sense.
\begin{prop}For $[a]\in \widetilde{R}(Y\times S^1)$, the restriction $[i^*a] \in \widetilde{R}(Y)$ satisfies 
\[
cs_Y([i^*a])=cs_{(Y\times S^1,\text{PD}([Y]))}([a],[\theta]),
\]
where $i$ is a inclusion $Y=Y\times 1 \ri Y\times S^1$ and \text{PD} is the Poincar\'e duality.
\end{prop}
This is a corollary of Lemma \ref{sumformula}. We have the following well-definedness.

\begin{lem}
$cs_{(X,\phi)}$ does not depend on the choices of $f$, representatives $a$ and $b$, and $A_{a,b}$.
\end{lem}
This is also a consequence of Lemma \ref{sumformula} 
\begin{defn}\upshape \label{tildeqx}Let $X$ be a closed oriented $4$-manifold equipped with $\phi \in H^1(X,\z)$. 
The invariant $\tilde{Q}_{(X,\phi)}$ is defined by
\begin{equation*}
\begin{cases}
\parbox{.9\linewidth}{%
$ \infty$     \ \ \ \ \                                          if $\widetilde{R}^*(X)= \emptyset$,
\\
$ \inf \left\{\left|cs_{(X,\phi)}([a],[\theta])+m\right| \middle| m\in \z\ , [a] \in \widetilde{R}^*(X) \right\} $ if $\widetilde{R}^*(X) \neq \emptyset$, 
}
\end{cases}
\end{equation*}
where $\widetilde{R}^*(X)$ is the subset of the classes of the irreducible connections in  $\widetilde{R}(X)$.
\end{defn}
We now give a definition of ${Q}^i_X \in \R_{\geq 0} \cup \{\infty\}$.
\begin{defn}\upshape \label{defofQiX}
Suppose that $X$ is a homology $S^3\times S^1$ and $i$ is a positive integer. Let $\widetilde{X}$ be $\z$-fold covering space over $X$ corresponding to the $1\in H^1(X,\z)\cong_{PD} H_3(X,\z)$.
We set $\widetilde{X}^i:= \widetilde{X}/ i\z$. Since the quotient map $p^i:\widetilde{X}\ri \widetilde{X}^i$ is a $\z$-fold covering, this determine a class $\phi^i \in H^1(\widetilde{X}^i,\z)$.
We define $Q^l_X\in \R_{\geq 0} \cup \{\infty\}$ by 
\[
Q^l_X:=\displaystyle \min_{0\leq i \leq l}\tilde{Q}_{(\widetilde{X}^i,\phi^i)}.
\]

\end{defn}
We show the following lemma which is used in the proof of Key lemma(Lemma \ref{lem:theta}).
\begin{lem}\label{lem:cs}
Suppose that $\gamma$ is a flat connection on $W[m,n]$ satisfying the following conditions.
\begin{itemize}
\item $\gamma|_{Y^m_+} \cong \gamma|_{Y^n_-}$.
\item There exists $u \in \z$ satisfying $\left|cs_{(\overline{W[m,n]},\text{PD}[Y^m_+])}(r(\gamma),\theta)+ u\right|< Q^{n-m+1}_X$, where $r(\gamma)$ is a flat connection on $\overline{W[0,k]}$ given by pasting $\gamma$ with itself along $Y^0_+ \cup Y^k_-$.
\end{itemize}
Then $\gamma$ is gauge equivalent to $\theta$.
\end{lem}
\begin{proof} Suppose $\gamma$ is not  gauge equivalent to $\theta$.
The calculation 
\[
H_1(W[m,n])\cong 0
\]
 and holonomy correspondence 
\[
R(W[m,n]) \cong \hom (\pi_1(W[m,n]),SU(2))/ \text{conjugate}
\]
imply that there is no reducible $SU(2)$ connection on $W[m,n]$ except $\theta$. Therefore $\gamma$ is an irreducible connection on $W[m,n]$.
Because $\overline{W[m,n]} \ri X$ is the $(n-m+1)$-fold covering space of $X$,
\[
Q^{n-m+1}_X \leq \left|cs_{(\overline{W[m,n]},\text{PD}[Y^m_+])}(r(\gamma),\theta)+ u\right|
\]
holds by the definition of $Q^i_X$. This is a contradiction.

\end{proof}
\section{Compactness}
The compactness of the instanton moduli spaces for non-compact 4-manifolds is treated in \cite{Fl88}, \cite{Fu90},\cite{Do02} for cylindrical end case and in \cite{T87} for periodic end case. In \cite{Fu90} and \cite{T87}, they consider the instanton moduli spaces with the connections asymptotically convergent to the trivial connection on the end.  
 We also follow their strategy by using $Q^{2l_Y+3}_X$ defined in the previous section. More explicitly, in this section we explain a compactness result for the instanton moduli spaces for a non-compact manifold $W[0,\infty]$ with periodic end.

\subsection{Key lemma}
Let $W_0$, $W[0,\infty]$ be the oriented Riemannian 4-manifolds introduced in the beginning of Subsection \ref{Fred}. By pasting $W_0$ with itself along its boundary $Y$ and $-Y$, we obtain a homology $S^3\times S^1$ which we denote by $X$. We consider the product $SU(2)$-bundle $P_{W[0,\infty]}$ on $W[0,\infty]$.
For $q\geq3$ and $\delta>0$, we define {\it the instanton moduli space} $M^{W[0,\infty]}_\delta$ by 
\[
M^{W[0,\infty]}_\delta := \left\{ \theta +c \in \Om^1(W[0,\infty])\otimes \su_{L^2_{q,\delta}} \middle | F^+(\theta+c)=0 \right\}/ \G,
\]
where $\G$ is the gauge group
\[
\G:=\left\{g \in \aut(P_{W[0,\infty]}) \subset \End(\mathbb{C}^2)_{L^2_{q+1,\text{loc}}} \middle | dg \in L^2_{q,\delta} \right\},
\]
and the action of $\G$ is given by the pull-backs of connections. For $f\in \Omega^i(W[0,\infty])\otimes \su$ with compact support, we define $L^2_{q,\delta}$ norm by the following formula
\[
||f||^2_{L^2_{q,\delta}}:= \sum_{0\leq j \leq q} \int_{W[0,\infty]} e^{\delta \tau}  \left|\nabla_{\theta}^j f\right|^2  d\text{vol} ,
\]
where $\nabla_{\theta}$ is the covariant derivertive with respect to the product connection. We use the periodic metric $|-|$ which is induced from the Riemannian metric $g_W$. Its completion is denoted by $\Omega^i(W[0,\infty])\otimes \su_{L^2_{q,\delta}}$.

 Our goal of this section is to show the next theorem under the above setting.
\begin{thm}\label{cptness}Under Assumption $\ref{imp}$ the following statement holds.
There exist $\delta'>0$ satisfying the following property. Suppose that $\delta$ is a non-negative number less than $\delta'$ and $\{A_n\} $ is a sequence in $M^{W[0,\infty]}_\delta$ satisfying 
\[
\displaystyle \sup_{n \in \n}||F(A_n)||^2_{L^2(W[0,\infty])} < \min\{8\pi^2,Q^{2l_Y+3}_X\}.
\]

Then for some subsequence $\{A_{n_j}\}$, a positive integer $N_0$ and some gauge transformations $\{g_j\}$ on $W[N_0,\infty]$, the sequences $\{g_j^*A_{n_j}\}$ converges to some $A_\infty$ in $L^2_{q,\delta}(W[N_0,\infty])$.
\end{thm}
The proof of Theorem \ref{cptness} is given in the end of Subsection \ref{c3}.


We use the following estimate.
\begin{lem}\label{lem:fundamental}
For a positive number $c_1>0$, there exists a positive number $c_2>0$ satisfying the following statement.

For any $SU(2)$-connection $a$ on $Y^0_+$ and any flat connection $\gamma$ on $W[0,k]$ satisfying the following conditions
\begin{itemize}
\item
$\sup_{x \in Y^+_0}\sum_{0\leq j \leq 1}\left|\nabla^{(j)}_{\gamma}(a-(l^0_+)^*\gamma)(x)\right|<c_1$.
\item $\gamma|_{Y^0_+ } \cong \gamma|_{Y^k_-}$.
\end{itemize}
, the inequality
\[
\left| cs_Y([a])-cs_{(\overline{W[0,k]},\text{PD}[Y^0_+])}(r(\gamma),\theta)\right| \leq c_2 \sup_{x \in Y^+_0}\sum_{0\leq j \leq 1}\left|\nabla^{(j)}_{(l^0_+)^*\gamma}(a-(l^0_+)^*\gamma)(x)\right|^2
\]
holds, where $r(\gamma)$ is a flat connection on $\overline{W[0,k]}$ given by pasting $\gamma$ with itself along $Y^0_+ \cup Y^k_-$ and $l^0_+:Y^0_+\ri W_0$ is the inclusion.
\end{lem}

\begin{proof}
Lemma \ref{sumformula} imply
\[
\left|  cs_Y([a])-cs_{(\overline{W[0,k]},\text{PD}[Y^0_+])}(\gamma,\theta)\right|
=|cs_Y(a)-cs_Y((l^0_+)^*\gamma)|.
\]
Since $(l^0_+)^*\gamma$ is a flat connection on $Y^0_+$, we have
\[
=\frac{1}{8\pi^2}\left| \int_{Y^0_+}Tr( (a-(l^0_+)^*\gamma) \wedge d_{(l^0_+)^*\gamma}(a-(l^0_+)^*\gamma) + \frac{2}{3}(a-(l^0_+)^*\gamma) ^3\right|
\]
\[
\leq \frac{1}{8\pi^2}\text{vol}_Y (\sup_{x \in Y^0_+}\left|\nabla_{(l^0_+)^*\gamma}(a-(l^0_+)^*\gamma) (x)\right||(a-(l^0_+)^*\gamma)|+ \frac{2}{3}\sup_{x \in Y^0_+}|(a-(l^0_+)^*\gamma)|^3)
\]
\[
\leq c_2 \sup_{x \in Y^0_+}\sum_{0\leq j \leq 1}|\nabla^{(j)}_{(l^0_+)^*\gamma}(a-(l^0_+)^*\gamma)(x)|^2.
\]

\qed
\end{proof}

Next lemma gives us a key estimate. We use $Q^{2l_Y+3}_X$ to obtain an estimate of the difference between an ASD-connection and the trivial flat connection on the end $W[0,\infty]$.
\begin{lem}\label{lem:theta}Suppose that $Y$ satisfies Assumption \ref{imp}. There exists a positive number $c_3$ satisfying the following statement.

For $A \in M^{W[0,\infty]}_{\delta}$ satisfying $\frac{1}{8\pi^2}||F(A)||^2_{L^2(W[0,\infty]}< \min \{1, Q^{2l_Y+3}_X\}$, there exists a positive number $\eta_0$ which depends only on the difference $\min \{1, Q^{2l_Y+3}_X\}-\frac{1}{8\pi^2}||F(A)||^2_{L^2}$ such that  the following condition holds.

Note that if $K$ is sufficiently large, the inequality $||F(A)||^2_{L^2(W_k)}< \eta_0$ is satisfied for every $k>K$. When $K$ satisfies this property, there exist gauge transformations $g_k$ over $W[k,k+2]$ such that 
\[
\sup_{x\in W[k,k+2]}\sum_{0\leq j\leq q+1}|{\nabla^{j}}_{\theta}({g_k}^*A|_{W[k,k+2]}-\theta)(x)|^2
\]
\[
\leq c_3{||F(A)||^2}_{L^2(W[k-l_Y-2,k+l_Y+3])}
\]
holds for $k>K+l_Y+3$.
\end{lem}
\begin{proof}
For $k>K+l_Y+3$, we apply Lemma $10.4$ in \cite{T87} to $A|_{{W[k-l_Y-1,k+l_Y+2]}}$. 
Then we obtain gauge transformations $g_k$ and flat connections $\gamma_k$ over $W[k-l_Y-1,k+l_Y+2]$ satisfying
\[
\sup_{x \in W[k-l_Y-1,k+l_Y+2]}\sum_{0\leq j\leq q+1}|{\nabla^{j}}_{\gamma}({g_k}^*A|_{W[k-l_Y-1,k+l_Y+2]}-\gamma_k)(x)|^2
\]
\begin{align}\label{glo}
\leq c_3{||F(A)||^2}_{L^2(W[k-l_Y-2,k+l_Y+3])}\leq (2l+5)c_3\eta
\end{align}
for a small $\eta$.
By using the pull-backs of $\gamma_k$ from $W[k-l_Y-1,k-1]$(resp. $W[k+2,k+l_Y+2]$) to $Y^i_+$, we get the flat connections over $Y$. Then we get $l_Y+1$ flat connections by using this method. Under the assumption that $l_Y=\# R(Y)$, same flat connections appear by the pigeonhole principle. We choose two numbers  $k(1)^\pm<k(2)^\pm$ which satisfy ${(l^{k(1)^\pm}_+)}^*\gamma_k \cong {(l^{k(2)^\pm}_+)}^*\gamma_k$ as connections on $Y$, where $k(1)^+$ and $k(2)^+$ are elements in $\{k-l_Y-1,\cdots , k-1\}$ (resp. $k(1)^-,k(2)^- \in \{k+2,k+l_Y+2\}$). The map $l^{k}_\pm:Y^k_\pm \ri W_k$ is the inclusion.

\begin{claim}
Suppose $||F(A)||^2_{L^2(W_k)}<\eta$ holds for $k>K+l_Y+3$. For sufficiently small $\eta$, the flat connection $\gamma_k$ is isomorphic to  $\theta$.
\end{claim}
The properties of $k(1)^\pm$ and $k(2)^\pm$, \eqref{glo} and Lemma $\ref{lem:fundamental}$ imply
\begin{align}\label{1}
\left| cs_Y((l^{k(1)^\pm}_+)^* {g_k}^*A) -cs_{(\overline{(W[k(1)^\pm,k(2)^\pm]},\text{PD}[Y^{k(1)^\pm}_+])}(r(\gamma_k),\theta)\right|\leq (2l+5)c_3\eta c_2.
\end{align}
We also have
\[
 cs_Y((l^{k(1)^\pm}_+)^* {g_k}^*A)= cs_Y(l^{k(1)^\pm}_+)^*A)+\text{deg}(g_k|_{Y^{k(1)^\pm}_+})
 \]
 \begin{align}\label{3}
 = ||F(A)||^2_{L^2(W[k(1)^\pm,\infty])}+\text{deg}(g_k|_{Y^{k(1)^\pm}_+}).
 \end{align}
 
We choose $\eta_0$ satisfying the following condition: 
\begin{align}\label{2}
(2l+5)c_3\eta_0 c_2 <  Q^{2l_Y+3}_X - \frac{1}{8\pi^2}\int_{W[k(1)^\pm,\infty]}|F(A)|^2,
\end{align}
where the right hand side is positive by the assumption of $A$.
We obtain 
\[
\left|cs_{(\overline{W[k(1)^\pm,k(2)^\pm]}, \text{PD}[Y^{k(1)^\pm}])}(r(\gamma_k),\theta)+\text{deg}(g_k|_{Y^{k(1)^\pm}_+})\right|< Q^{2l_Y+3}_X
\]
 from \eqref{1}, \eqref{3} and \eqref{2}. Then Lemma \ref{lem:cs} imply $\gamma|_{W[k(1)^\pm,k(2)^\pm]}\cong \theta$.

Similarly the inequality
\[
\sup_{x \in W[k(2)^-,k(1)^+]}\sum_{0\leq j\leq q}|{\nabla^{j}}_{\gamma_k}({g_k}^*A|_{\widetilde{W}}-\gamma_k|_{W[k(2)^-,k(1)^+]})(x)|\leq (2l_Y+5)c_3\eta_0
\]
holds over $W[k(2)^-,k(1)^+]$.
From above discussion, ${l^{k(2)^-}_+}^*\gamma_k$ and ${i^{k(1)^+}_+}^*\gamma_k$ are gauge equivalent.
By Lemma \ref{lem:fundamental}, we also get 
\[
\left| cs_Y((l^{k(2)^-}_+)^*g_k^*A)-cs_{(\overline{W[k(2)^-,k(1)^+]}, \text{PD}[Y^{k(2)^-}_+])}(r(\gamma_k),\theta)\right|\leq (2l_Y+5)c_3\eta_0c_2.
\]
By choice of $\eta_0$, we have 
\[
\left|cs_{(\overline{W[k(2)^-,k(1)^+]}, \text{PD}[Y^{k(2)^-}_+])}(r(\gamma_k),\theta)+ \text{deg}(g_k|_{Y^{k(2)^-}_+})\right|< Q^{2l_Y+3}_X
\]
and get $\gamma_k|_{\overline{W[k(2)^-,k(1)^+]}}\cong \theta$ by using Lemma \ref{lem:cs} as above.
\qed
\end{proof}

\subsection{Chain convergence}
We introduce the following notion which is crucial for our proof of the compactness theorem(Theorem \ref{cptness}).
\begin{defn} \upshape [Chain decomposition of a sequence of the instantons]
For a fixed number $\eta>0$ and a sequence $\{A_n\}\subset M^{W[0,\infty]}_{\delta}$ satisfying 
\[
\displaystyle \sup_{n\in \n}||F(A_n)||^2_{L^2(W[0,\infty])}<\infty,
\]
when a finitely many sequences $\{s^j_n(\eta)\}_{1\leq j \leq m}$ of non-negative numbers satisfies
\[ 
{||F(A_n)||^2}_{L^2(W_{s})}>\eta\ \iff \ s=s^j_n \text{ for some }j,
\]
and
\[
s^1_n < \dots < s^m_n\ ,
\]
we call $\{s^j_n\}_{1\leq j \leq m}$ {\it chain decomposition} of $\{A_n\}\subset M^{W[0,\infty]}_{\delta}$ for $\eta$. 
\end{defn}
\begin{rem}\upshape
For any sequence $\{A_n\}$ satisfying $\sup_{n\in \n}||F(A_n)||^2_{L^2(W[0,\infty])}<\infty$ and any $\eta>0$, we can show the existence of a chain decomposition for $\eta>0$ if we take a subsequence of $\{A_n\}$.
\end{rem}
First we give the next technical lemma.
\begin{lem}\label{idhom}
Let $A$ be a $L^2_q$ ASD-connection on $W[k,\infty]$ satisfying 
\[
\displaystyle \frac{1}{8\pi^2} ||F(A)||^2_{L^2(W[k,\infty])}<1.
\]
Then there exists a positive number $c_4$ which depends only on the difference $1-\frac{1}{8\pi^2} ||F(A)||^2_{L^2(W[k,\infty])}$  such that the following statement holds.

Suppose there exists a gauge transformation $g$ on $Y^k_+$ satisfying 
\[
\sum_{0\leq j\leq 1}\sup_{x \in Y^k_+}   \left| \nabla^j_\theta (g^*(l^k_+)^*A-\theta)(x) \right|^2 \leq c_4.
\]
Then $g$ is homotopic to the identity gauge transformation.
\end{lem}

\begin{proof}
By the property of $cs_Y$, we have
\[
|\text{deg}(g)|= |cs_Y(g^*(l^k_+)^*A)-cs_Y((l^k_+)^*A)|
\]
\[
\leq |cs_Y(g^*(l^k_+)^*A)|+||F(A)||^2_{L^2(W[k,\infty])}
\]
\[
\leq  \sum_{0\leq j\leq 1}\sup_{x \in Y^k_+}   \left| \nabla^j_\theta (g^*(l^k_+)^*A-\theta)(x)\right|^2+||F(A)||^2_{L^2(W[k,\infty])}.
\]
We define $c_4< \frac{1}{2}(1-||F(A)||^2_{L^2(W[k,\infty])})$, then 
\[
|\text{deg}(g)| < 1
\]
holds. This implies the conclusion.

\end{proof}

The next proposition is a consequence of Lemma \ref{lem:theta} and \ref{idhom}.
\begin{prop}\label{lem:type}
Let $\eta$ be a positive number and $\{s^j_n\}_{1\leq j \leq m}$ be a chain decomposition for $\eta$ of a sequences $\{A_n\}$ in $M^{W[0,\infty]}_{\delta}$ with 
\[
\displaystyle \frac{1}{8\pi^2}\sup_{n \in \n}||F(A_n)||^2_{L^2(W[0,\infty])} <\min \{1,Q^{2l_Y+3}_X\}.
\]
Then there exists a subsequence $\{A_{n_i}\}$ of $\{A_n\}$ such that $\sup_{i\in \n}  |s^m_{n_i}|<\infty$ holds.

\end{prop}
\begin{proof}
Suppose there exists $\eta_0>0$ which does not satisfy the condition of Proposition\ $\ref{lem:type}$. There exist a chain decomposition $\{s^j_n(\eta_0) \}_{0\leq j\leq m}$ of $\{A_n\}$ which satisfies $s^m_{n} \ri \infty$ as $n \ri \infty$. We take a sufficiently small $\eta>0$, we will specify $\eta$ later. Choose a subsequence of $\{A_n\}$ which allows a chain decomposition for $\eta$. For simplify, we denote the subsequence by same notation $\{A_n\}$. We denote the chain decomposition of $\{A_n\}$ for $\eta$ by  $\{t^{j}_n(\eta)\}_{1\leq j \leq m'}$.
There are two cases for $\{t^{j}_n(\eta)\}_{1\leq j \leq m'}$:
\begin{itemize}
\item There exists $j'' \in \{0,\cdots, m'-1\}$ satisfying $t^{j''}_n - t^{j''+1}_n\ri -\infty$.
\item There is no $j'' \in \{0,\cdots, m'-1\}$ satisfying $t^{j''}_n - t^{j''+1}_n\ri -\infty$.
\end{itemize}
We define a sequence by 
\[
u_n(\eta):=\left\lfloor \frac{t^{j''}_n(\eta)+t^{j''+1}_n(\eta)}{2} \right\rfloor \in \n
\]
for the first case and by
\[
u_n(\eta):=0
\]
for the second case, where $\lfloor - \rfloor$ is the floor function.
Applying Lemma \ref{lem:theta} to $A_n|_{W[u_n(\eta),u_n(\eta)+2]}$, we get the gauge transformation $g_n$ on $W[u_n(\eta)-l_Y-2,u_n(\eta)+l_Y+3]$ satisfying 
\begin{align}\label{yyy}
\sup_{x \in W[u_n(\eta),u_n(\eta)+2]}\sum_{1\leq j \leq q+1}|\nabla^j_{\theta}({g_n}^*{A_n}|_{W[u_n,u_n(\eta)+2]} - \theta)|^2\leq c_3(2l_Y+5)\eta,
\end{align}
for small $\eta$ and large $n$. 
 Because $A_n$ is the ASD-connection for each $n$, we have
\[
\frac{1}{8\pi^2}\int_{W[u_n,\infty]}Tr(F(A_n)\wedge F(A_n))=\frac{1}{8\pi^2}\int_{W[u_n,\infty]}|F(A_n)|^2>\eta_0
\]
for large $n$.
On the other hand, by the Stokes theorem
\[
\frac{1}{8\pi^2}\int_{W[u_n,\infty]}Tr(F(A_n)\wedge F(A_n))=cs_Y({l^{u_n}_+}^*A_n)
\]
holds. 
By Lemma \ref{idhom}, we have 
\[
cs_Y({l^{u_n}_+}^*A_n)=cs_Y({l^{u_n}_+}^*g_n^* A_n)
\]
for small $\eta$. Therefore \eqref{yyy} and Lemma \ref{lem:fundamental} gives 
\[
|cs_Y({l^{u_n}_+}^*A_n)|\leq c'c_3(2l_Y+5)\eta.
\]
We choose $\eta$ satisfying
\[
c'c_3(2l_Y+5)\eta<\frac{1}{8\pi^2}\eta_0.
\]
For such $\eta$, we have \[
|cs_Y({l^{u_n}_+}^*A_n)|<\frac{1}{8\pi^2}\eta_0.
\]
On the other hand, $\eta_0$ satisfies \[
\frac{1}{8\pi^2}\eta_0 <\frac{1}{8\pi^2}\int_{W[u_n,\infty]}|F(A_n)|^2=|cs_Y({l^{u_n}_+}^*A_n)|
\] which is a contradiction.
\qed
\end{proof}

\subsection{Exponential decay}\label{c3}
In the instanton Floer theory, there is an estimate called exponential decay about the $L^2$-norm of curvature of the instanton over cylindrical end.  We give a generalization of the exponential decay estimate over $W[0,\infty]$. In the end of this subsection, we also give a proof of Theorem \ref{cptness}.

\begin{lem}\label{lem:estimate}
There exists a constant $c_5$ satisfying the following statement.
For $A \in M^{W[0,\infty]}_{\delta}$ satisfying $\frac{1}{8\pi^2}||F(A)||^2<\min \{1,Q^{2l_Y+3}_X\}$, there exists $\eta_1>0$  which depends only on the difference  $\min\{ Q^{2l_Y+3}_X,1\} -\frac{1}{8\pi^2}||F(A)||^2$ such that the following condition holds.

Let $K>0$ be a positive number satisfying $||F(A)||^2_{L^2(W_k)}< \eta_1$ for any $k>K$, the inequality 
\begin{align}\label{ooo}
||F(A)||^2_{L^2(W[k,k+m])} 
\end{align}
\[
\leq c_3( ||F(A)||^2_{L^2(W[k-l_Y-2,k+l_Y+3])}+||F(A)||^2_{L^2(W[k+m-l_Y-2,k+m+l_Y+3]})
\]
holds for $k>K+l_Y+3$.
\end{lem}
\begin{proof}
Let $\eta$ be the positive number in Lemma \ref{lem:theta} which depends only the difference $Q^{2l_Y+3}_X-\frac{1}{8\pi^2}||F(A)||^2$. Then for $k>K+l_Y+3$, we have the following inequalities
\[
\sup_{x \in W_k}\sum_{0\leq j\leq q}\left|{\nabla^{j}}_{\theta}(g_k^*A-\theta)(x)\right|^2\leq c_3||F(A)||^2_{L^2(W[k-l_Y-2,k+l_Y+3])}
\]
\begin{align}\label{i1}
\leq c_3(2l_Y+5)\eta_1
\end{align}
and
\[
\sup_{x \in W_{k+m}}\sum_{0\leq j\leq q}\left|{\nabla^{j}}_{\theta}(g_{k+m}^*A-\theta)(x)\right|^2\leq  c_3||F(A)||^2_{L^2(W[k+m-l_Y-2,k+m+l_Y+3]}
\]
\begin{align}\label{i2}
\leq c_3(2l_Y+5)\eta_1.
\end{align}
These inequality \eqref{i1}, \eqref{i2} and Lemma \ref{idhom} imply that for sufficiently small $\eta_1$, the gauge transfromation $g_k|_{Y^+_k}$(resp. $g_{k+m}|_{Y^-_{k+m}}$) is homotopic to the constant gauge transformation.
Hence, there exists a gauge transformation $\hat{g}$ on $W[k,k+m]$ satisfying $\hat{g}|_{W_k}=g_k$ and $\hat{g}|_{W_{k+m}}=g_{k+m}$, moreover, since $A$ is the ASD connection, we have
\[
||F(A)||^2=||F(\hat{g}^*A)||^2_{L^2(W[k,k+m])}=8\pi^2 (cs_Y({(l^k_+)}^*{g_k}^*A))-cs_Y({(l^{k+m}_-)}^*{g_{k+m}}^*A)).
\]
 Applying the inequalities \eqref{i1} and \eqref{i2} again, we get the conclusion.
\qed
\end{proof}

\begin{prop}[Exponential decay]\label{prop:expdecay}There exists $\delta'>0$ satisfying the following statement.

Suppose $A$ is an element in $M^{W[0,\infty]}_{\delta}$ satisfying the assumption of Lemma \ref{lem:estimate}.
Then there exists $c_5(K)>0$ satisfying the following inequality. 
\[
{||F(A)||^2}_{W[k-l_Y-2,k+l_Y+3]} \leq c_5(K)e^{-k\delta'}.
\] 
for $k>K+l_Y+3$.
\end{prop}

\begin{proof}
This is a consequence of Lemma $\ref{lem:estimate}$ and Lemma $5.2$ in \cite{Fu90} by applying $q_i = ||F(A)||^2_{W[i-l_Y-2,i+l_Y+3]}$.
\qed
\end{proof}

By using a similar argument in Lemma $4.2$ and Lemma $7.1$ of \cite{Fu90} we have:
\begin{lem}[Patching argument]\label{lem:patchingarg}
For a positive number $c_7$, there exists a constant $c_8$ satisfying the following statement holds.

Suppose we have an $L^2_q$ connection $A$, the gauge transformations $g_k$ on $W[k-1,k+1]$ satisfying
\[
\int_{W[k-1,k+1]}\sum_{0\leq j\leq q+1}|{\nabla^{j}}_{\theta}({g_k}^*A|_{W[k-1,k+1]}-\theta)|^2
\]
\[
\leq c_7{||F(A)||^2}_{L^2(W[k-l-3,k+l+2])}.
\]
for any non-negative integer $k$.

Then there exists the positive integer $n_0$ and a gauge transformation $g$ on $W[n_0,\infty] $ satisfying the following condition:
\[
\int_{W[k-1,k+1]}\sum_{0\leq j\leq q+1}|{\nabla^{j}}_{\theta}({g}^*A|_{W[k-1,k+1]}-\theta)|^2
\]
\[
\leq c_8{||F(A)||^2}_{L^2(W[k-l-3,k+l+2])}
\]
for $k>n_0$.
\end{lem}
We use this lemma to prove the next proposition.
\begin{prop}\label{prop:conv}
There exists $\delta'>0$ satisfying the following condition.
Let $K>0$ be a positive number and $\{A_n\}$ be a sequence in $M^{W[0,\infty]}_\delta$ satisfying the following properties:
\begin{itemize}
\item  $\displaystyle 0< \min \{Q^{2l_Y+3}_X,1\}- \sup_{n\in \n } \frac{1}{8\pi^2}||F(A_n)||^2$.
\item
There exists a chain decomposition $\{s^n_j\}$ of $\{A_n\}$ for $\eta_2$ satisfying 
\[
\sup_{n\in \n }|s^m_n(\eta_*)|<\infty
\]
  for 
\[
 \eta_2:= \inf_{n\in \n}\left\{\eta |\text{ constants which depend on $A_n$ in Lemma \ref{lem:theta} and \ref{lem:estimate}} \right\}.
 \]
\end{itemize}

Then there exist a positive integer $N_0$, gauge transformations $\{g_j\}$ on $W[N_0,\infty]$ and subsequence $\{A_{n_j}\}$ of $\{A_n\}$ such that $\{{g_j}^*A_{n_j}\}$ converge to some $A_\infty$ in $L^2_{q,\delta}(W[N_0,\infty])$ for any $0\leq \delta<\delta'$.
\end{prop}

\begin{proof}
If we apply the Lemma $\ref{lem:theta}$ to $A_n$, there exists gauge transformations $g^n_k$ on $W[k-1,k+1]$ satisfying the following condition: for $k>l_Y+K+3$,
\[
\int_{W[k-1,k+1]}\sum_{0\leq j\leq q+1}|{\nabla^{j}}_{\theta}({g^n_k}^*A_n|_{W[k-1,k+1]}-\theta)|^2
\]
\[
\leq c_3{||F(A_n)||^2}_{L^2(W[k-l-3,k+l+2])}.
\]

On the other hand, we have 
\begin{align}\label{exx}
{||F(A_n)||^2}_{L^2(L^2(W[k-l-3,k+l+2])} \leq c_6(K)e^{-\delta' k}
\end{align}
by using the exponential decay estimate(Proposition \ref{prop:expdecay}).
Using \eqref{exx}, we can show that $n_0$ uniformly with respect to $n$ in Lemma \ref{lem:patchingarg}.
So there exist a large natural number $N_0$ and a gauge transformation on $W[N_0,\infty]$ for each $n$ satisfying
\[
\int_{W[k-1,k+1]}\sum_{0\leq j\leq q+1}|{\nabla^{j}}_{\theta}({g_n}^*A_n|_{W[k-1,k+1]}-\theta)|^2
\]
\begin{align}\label{o}
\leq c'_8{||F(A_n)||^2}_{L^2(W[k-l-3,k+l+2])} \leq c_6(K)c'_8e^{-\delta'k},
\end{align}
where the last inequality follows from \eqref{exx}.

We set $g_n^*A_n=\theta +a_n$. Then we have
\[
||a_n||_{{L^2_{q+1,\delta}(W[N_0,\infty])}}=\sum_{0\leq j\leq q+1}\int_{W[N_0,\infty] } e^{\delta \tau}|{\nabla^{j}}_{\theta}(a_n)|^2
\]
\[
\leq \sum_{0\leq j\leq q+1}\sum_{N_0 \leq i\leq \infty} e^{i\delta } \int_{ {W_i}} |{\nabla^{j}}_{\theta}(a_n)|^2.
\]
Putting this estimate and \eqref{o} together, we have 
\begin{align}\label{expdecay}
||a_n||^2_{{L^2}_{q+1,\delta}(W[k,\infty] )}\leq c_{9}e^{(\delta-\delta')k}
\end{align}
for $k>N_0$.
 We take a subsequence of $\{a_n\}$ which converges on any compact set in $L^2_q(W[k,\infty]) $ by using the Relich Lemma. We denote the limit in $L^2_{q,\text{loc}}$ by $a_\infty$.
Then the exponential decay \eqref{o} and a standard argument implies that $\{a_n\}$ converges $a_\infty$ on $W[N_0,\infty]$ in $L^2_{q,\delta}$-norm. 
\qed
\end{proof}
We now give the proof of Theorem \ref{cptness}.

\begin{proof}
We choose $\eta_2$ in Proposition \ref{prop:conv}.  After taking subsequence of $\{A_n\}$, we consider the chain decomposition $\{s^j_n\}_{1\leq j \leq m} $ for $\eta_2$ of $\{A_n\}$. 
From Proposition \ref{lem:type}, $\{s^j_n\}$ has upper bound by some $K>0$ after taking a subsequence of $\{A_n\}$ again. So we can apply Proposition \ref{prop:conv}, we get the conclusion.
\qed
\end{proof}

\section{Perturbation and Orientation}
To prove the vanishing $[\theta^r]=0$ in Theorem \ref{mainthm}, we use the moduli spaces $M^W(a)_{\pi,\delta}$ and need the transversality for the equation $F^+(A)+s\pi(A)=0$. We also need the orientability of $M^W(a)_{\pi,\delta}$.

\subsection{Holonomy perturbation(2)}
In \cite{Do87}, Donaldson introduced the holonomy perturbation with compact support for irreducible ASD-connections. Combining the technique in \cite{Do87} and the compactness theorem (Theorem \ref{cptness}), we get sufficient perturbations to achieve required transvesality.  
\begin{defn} \upshape \label{hol2}
Let  $\pi$ be an element in $\prod(Y)$ and $a$ be a critical point of $cs_{Y,\pi}$.
We use the following notations: 
\begin{itemize}
\item $\Gamma (W):=\left\{  l:S^1 \times D^3 \ri W \middle | \text{$l$: orientation preserving embedding} \right\}.$
\item $\Lambda^d(W):= \left\{(l_i, \mu^+_i)_{1\leq i \leq d} \in \Gamma(W)^d\times (\Om^+(W)\otimes \su))^d \middle | \text{supp} \mu^+_i \subset \im l_i \right\}$.
\item $\displaystyle \Lambda(W):= \bigcup_{d\in \n } \Lambda^d(W)$.
\end{itemize}

Let $\chi:SU(2)\ri \mathfrak{su}(2)$ be
\begin{align*}
\chi(u):=u-\frac{1}{2}tr(u)id
\end{align*}
and fix $\mu^+_i \in \Omega^+(W)\otimes \mathfrak{su}(2)$ supported on $ l_i(S^1\times D^3)$ for $i\in \{1,\cdots ,d\}$.
For $\epsilon \in \R^d$, we set
\begin{align*}
\sigma_{\Psi}(A,\epsilon):= \sum_{1\leq i \leq d} \epsilon_i  \chi(\hol_{x \in l_i(S^1\times D^3)} (A))\mu^+_i,
\end{align*}
where $\hol_{x\in l_i(S^1\times D^3)}$ is a holonomy around the loop $t \mapsto l_i(t,y_x)$ satisfying $x=l_i (t_x,y_x)$ for some $t_x$ and $\epsilon=(\epsilon_i)_{1\leq i \leq d}$.
For $\Psi=(l_i,\mu_i)_{1\leq i \leq d} \in \Lambda$, Donaldson defined {\it the holonomy perturbation of the ASD-equation}:
\begin{align}\label{pert}
\mathcal{F}_{\pi,\Psi}(A,\epsilon):=F^+(A)+s\pi(A)+\sigma_\Psi(A,\epsilon)=0.
\end{align}
The map $\sigma_\Psi(-,\epsilon)$ is smoothly extended to the map $\A^W(a)_{\delta}\ri \Omega^+(W)\otimes \mathfrak{su}(2)_{L^2_{q-1,\delta}}$ and the map $\A^W(a)_{(\delta,\delta)}\ri \Omega^+(W)\otimes \mathfrak{su}(2)_{L^2_{q-1,(\delta,\delta)}}$. For $\Psi$ and $\epsilon \in \R^d$, {\it the perturbed instanton moduli space} are defined by 
\[
M^W(a)_{\pi,\Psi,\epsilon,\delta}:=\left\{ c \in \B^W(a)_{\delta} \middle |\mathcal{F}_{\pi,\Psi}(c,\epsilon)=0 \right\} 
\]
in the case of $a\in \widetilde{R}^*(Y)_\pi$ and
\[
M^W(a)_{\pi,\Psi,\epsilon,(\delta,\delta)}:=\left\{ c \in \B^W(a)_{(\delta,\delta)} \middle |\mathcal{F}_{\pi,\Psi}(c,\epsilon)=0 \right\}
\]
in the case of $\text{Stab}(a)=SU(2)$.
For a fixed $\epsilon \in \R^d$, if the operator 
\[
d(\mathcal{F}_{\pi,\Psi})_{(A,0)}:T_A\A^W(a)_{\delta}\times \R^d \ri \Omega^+(W)\otimes \su_{L^2_{q-1,\delta}}.
\]
 is surjective for all $[A] \in M^W(a)_{\delta,\pi,\Psi,\epsilon}$, we call $(\Psi,\epsilon)$ {\it regular perturbation} for $a$$\in \widetilde{R}^*(Y)_\pi$.

Let $FM^W(a)_{\delta,\pi,\Psi}$ be {\it the family version of the perturbed instanton moduli spaces} defined by
\[
FM^W(a)_{\delta,\pi,\Psi}:=\left\{ (c,\epsilon) \in \B^W(a)_{\delta}\times \R^d \middle |\mathcal{F}_{\pi,\Psi}(c,\epsilon)=0 \right\}.
\]

\end{defn}
\begin{lem}\label{lem:Sur}Suppose that $Y$ satisfies Assumption \ref{imp}.
There exists $\delta'>0$ such that for a fixed $\delta \in (0,\delta')$, the following statement holds.
Suppose $\pi$ is a holonomy perturbation which is non-degenerate and regular. Let $a$ be an irreducible critical point of $cs_{Y,\pi}$ with $cs_Y(a)<\min\{Q^{2l_Y+3}_X,1\}$.
We assume the next three hypotheses for $(\pi,a)$.
\begin{enumerate}
\item  For $[A] \in M^W(b)_{\pi,\delta}$,
\[
\displaystyle \frac{1}{8\pi^2}\sup_{n\in \n} ||F(A)+s\pi(A)||^2_{L^2(W)} <\min \{1 ,Q^{2l_Y+3}_X\},
\]
 where $b$ is an element of  $\widetilde{R}(Y)_\pi$ with  $cs_{Y,\pi} (b) \leq cs_{Y,\pi}(a)$. 
\item The linear operator
\[
d^+_\theta+  sd\pi^+_\theta:T_\theta\A^W(\theta)_{(\delta,\delta)}\times \R^d \ri \Omega^+(W)\otimes \su_{L^2_{q-1,(\delta,\delta)}}
\]
is surjective.
\item $M(c)_{\pi,\delta}$ is empty set for $c\in \widetilde{R}_\pi(Y)$ satisfying $cs_{Y,\pi}(c)<0$.
\end{enumerate}

Then there exist a small number $\eta>0$ and a perturbation $\Psi$ such that the map
\[
d\mathcal{F}_{\pi,\Psi} :T\A^W(a)_{\delta}\times \R^d \ri \Omega^+(W)\otimes \su_{L^2_{q-1,\delta}}.
\]
is surjective for all point in ${\mathcal{F}_{\pi,\Psi}}^{-1}(0)\cap (\A^W(a)_{\delta}\times B^d(\eta))$.

\end{lem}
\begin{proof}
First we show that the surjectivity of $d\mathcal{F}_{\pi,\Psi}$ at the point in  ${\mathcal{F}_{\pi,\Psi}}^{-1}(0)\cap (\A^W(a)_{\delta}\times \{ 0\} )$.
Second, we show that there exists a positive number $\eta>0$ such that $d\mathcal{F}_{\pi,\Psi}$ is surjectivie at the point in  ${\mathcal{F}_{\pi,\Psi}}^{-1}(0)\cap (\A^W(a)_{\delta}\times B^d(\eta))$.
We name the critical point of $cs_{Y,\pi}$ by
\[
0=cs_{Y,\pi}( \theta=a_0 ) \leq cs_{Y,\pi}(a_1)\leq cs_{Y,\pi} (a_2 ) \cdots  \leq cs_{Y,\pi}(a_w=a).
\]
The proof is induction on $w$ and there are four steps.  
\begin{step}\label{step1}  For an irreducible element $A \in \A(a_{w})_{\delta}$ with $0\neq  \coker (d(\mathcal{F}_{\pi,\Psi})_{(A,0)})$, there exists $\Psi(A)=\{l^{A}_i,\mu^{A}_i \}_{1\leq i \leq d(A)}$ such that $d(\mathcal{F}_{\pi,\Psi})_{(A,0)}| \R^{d(A)}$ generates the space \coker $(d^+_A+sd\pi^+ _A)$.
\end{step}
The proof is essentially the same discussion of Lemma $2.5$ in \cite{Do87}. 
We fix $h \in \Om^+(W)\otimes \su_{L^2_{q-1,\delta}}$ satisfying $0\neq h \in \coker (d^+_A+sd\pi _A$) with $||h||_{L^2}=1$. The unique continuation theorems: Proposition $8.6$ (ii) of \cite{SaWe08} for the equation $(d^+_A(-)+d\pi^+_A)^*(-)=0$ on $Y\times (-\infty,-1]$ and Section $3$ of \cite{FU91} for the equation $(d^+_A)^*(-)=0$ on $W[0,\infty]$ imply $h|_{Y\times [-2,0] \cup_Y W_0}\neq 0$.
Then we choose $x_h$ in $Y\times [-2,0] \cup_Y W_0$ so that $h(x_h)\neq0$ holds.
Since $A$ is the irreducible connection, 
\[
\hol(A,x_h)=\left\{ \hol_l(A) \in SU(2) \middle |l \text{: loop based at } x_h \right\}
\]
 is a dense subset of $SU(2)$. 
So we can choose the loops $l^{h}_i$ based at $x_h$ satisfying 
\[
\{e_i=\chi(\hol_{l^{h}}(A))\}_{i} \text{ generates }\su.
\]
 For a small neighborhood $U_{x_h}$ of $x_h$, we can write $h$ by
\[
h|_{U_{x_h}}= \sum_{1\leq i \leq 3}h_i \otimes e_i.
\]
By using a smoothing of $\delta$ function, we have
\begin{align}\label{point}
\inner<h|U_{x_h},\sum_{1\leq i \leq 3}\mu^+_i(h)\otimes \chi(Hol_{l_i}(A))>_{L^2(U_{x_h})} \neq 0,
\end{align}
where $\mu^+_i(h)$ are three self dual $2$-forms supported on $U_{x_h}$. For a fixed generator $\{h^1,\dots ,h^u\}$ of \coker $(d^+_A+sd\pi^+_A)$, we get the points $\{x_{h^j}\} \subset Y\times [-2,0] \cup_Y W_0$, small neighborhoods $\{U_{x_{h^j}}\}$, loops $\{l^{h^j}_i\}$ and self dual 2-forms $\{\mu^+_i(h^j)\}$ satisfying \eqref{point} for all $h^j$. We extend the maps $l^{h^j}_i:S^1\ri W$ to embeddings $S^1\times D^3 \ri W$. We can choose $U_{x_{h^j}}$ satisfying $U_{x_{h^j}} \subset \im\ l^{h^j}$.  We set 
\[
\Psi(A):= (l^{h^j}_i,\mu^+_i(h^j))_{i,j} \in \Lambda(W),
\]
which satisfies the statement of Step \ref{step1}.

For $j\geq0 $ satisfying $cs_{Y,\pi}(w_j)=0$, we show:
\begin{step}\label{step2}
For an element $b \in \widetilde{R}^*(Y)_\pi$$($resp. $b=\theta$$)$  satisfying $cs_{Y,\pi}(b)=0$, there exists a perturbation $\Psi^b$ such that the operator
\[
d\mathcal{F}_{\pi,\Psi^b}|_{(A,0)}:T_A\A^W(b)_\delta \times \R^d\ri \Om^+(W)\otimes \su_{L^2_{q-1,\delta}}
\]
\[
(\text{resp. } d\mathcal{F}_{\pi,\Psi^b}|_{(A,0)}:T_A\A^W(b)_{(\delta,\delta)} \times \R^d\ri \Om^+(W)\otimes \su_{L^2_{q-1,(\delta,\delta)}} \text{})
\]
is surjective for $A \in (\mathcal{F}_{\pi,\Psi^b})^{-1}(0) \cap (\A^W(a)_{\delta}\times \{ 0\})$ $($resp. $A \in (\mathcal{F}_{\pi,\Psi^b})^{-1}(0) \cap (\A^W(a)_{(\delta,\delta)}\times \{ 0\})$$)$.
\end{step}
First we show that $M^W(b)_{\pi,\delta}$ is compact. Let $\{[A_n]\}$ be any sequence in $M^W(b)_{\pi,\delta}$. By the second hypothesis, we have 
\[
\frac{1}{8\pi^2}\sup_{n \in \n}{||F(A_n)||^2_{L^2(W[0,\infty])}} \leq \frac{1}{8\pi^2}\sup_{n \in \n}||F(A_n)+s\pi(A_n)||^2_{L^2(W)}<\min \{1 ,Q^{2l_Y+3}_X\}.
\]
By Theorem $\ref{cptness}$, there exist a large positive number $N$ and the gauge transformations $\{g_n\}$ over $W[N,\infty]$ such that $\{g_n^*A_n\}$  converges  over $W[N,\infty]$ for small $\delta$ after taking a subsequence. Note that $Y\times (-\infty,0] \cup_Y W[0,N+1]$ is a cylindrical end manifold and we can apply the general theory developed on Section $5$ of \cite{Do02}. In particular, there exist gauge transformations $\{h_n\}$ on $Y\times (-\infty,0] \cup_Y W[0,N+1]$ such that $\{h_n^*A_n|_{Y\times (-\infty,0] \cup_Y W[0,N+1]}\}$ has a chain convergent subsequence in the sense in Section $5$ in \cite{Do02} because the bubble phenomenon does occur under the first hypothesis
\[
\frac{1}{8\pi^2}\sup_{n \in \n}||F(A_n)+s\pi(A_n)||^2_{L^2(W)}<1.
\]
By gluing $\{g_n\}$ and $\{h_n\}$, we obtain a chain convergent subsequence
\[
[A_{n_j}] \ri ([C^1],\dots,[C^N], [A^0]) \in M(b=c_1,c_2)_\pi \times \dots \times M(c_v,c_{v+1})_\pi \times M^W(c_{v+1})_{\pi,\delta}
\]
 with $c_i \in \widetilde{R}(Y)_\pi$ .  Suppose that $[A_{n_j}] \ri [A^0] \in M(b)_{\pi,\delta}$ does not hold. We get $cs_{Y,\pi}(c_{v+1})<0$ because the moduli spaces 
 \[
 M(b,c_1)_\pi,\cdots , M(c_v,c_{v+1})_\pi
 \]
  are non-empty sets. However this contradicts to the assumption of $M(c)_{\pi,\delta}=\emptyset$ for $c\in \widetilde{R}(Y)_\pi$ with $cs_{Y,\pi}(c)<0$.

When $b$ is an irreducible connection, the compactness of $M(b)_{\pi,\delta}$, Step \ref{step1} and the openness of surjective operators imply Step \ref{step2}. When $b$ is equal to $\theta$, the second hypothesis implies Step \ref{step2}.

For the inductive step, we show:
\begin{step}\label{k}
Suppose there is a perturbation 
\[
\Psi^{w-1}=(l^{w-1}_i, \mu^{w-1}_i)_{i} \in \Lambda(W)
\]
 such that the operators \[
 d\mathcal{F}_{\pi,\Psi^{w-1}}|_{(A,0)}:T_A\A^W(a_j)_\delta \ri \Om^+(W)_{L^2_{q-1,\delta}}
 \]
 is surjective for $(A,0) \in (\mathcal{F}_{\pi,\Psi^{w-1}})^{-1} \cap (\A^W(a_j) \times \{0\})$ and $j \in \{1, \cdots, w-1\}$.
Then the space
\[
K_w:=\left\{A\in M^W(a_w)_{\pi,\delta} \middle |\ 0\neq  \coker (d\mathcal{F}_{\pi,\Psi^{w-1}}|_{(A,0)}) \subset \Omega^{+}(W)_{L^2_{q-1,\delta}} \right\}
\]
is compact.
\end{step}
Let $\{[A_n]\}$ be a sequence in $K_w$.
By the similar estimate in Step \ref{step2} and Theorem \ref{cptness}, we get a chain convergent subsequence
\[
[A_{n_j}] \ri ([B^1],\dots,[B^N], [A^0]) \in M(a_w=b_1,b_2)_\pi \times \dots \times M(b_v,b_{v+1})_\pi \times M^W(b_{v+1})_{\pi,\delta}
\]
 with $b_i \in \widetilde{R}(Y)_\pi$. Suppose that $[A_{n_j}] \ri [A^0] \in M(a_{w})_{\pi,\delta}$ does not hold.  In this case, the operators $d^+_{B^i}+d\pi_{B^i}$ on $Y\times \R$ and the operators $d\mathcal{F}_{\pi,\Psi^{w-1}}|_{(A^0,0)}$ on $W$ are surjective in the suitable functional spaces by the assumption of $\pi$ and the induction. For large $j$, the operator $d\mathcal{F}_{\pi,\Psi^{w-1}}|_{(A^{n_j},0)}$ can be approximated by the gluing of the operators $d_{B^i}^++d\pi_{B^i}$, $d\mathcal{F}_{\pi,\Psi^{w-1}}|_{(A^0,0)}$. By gluing the right inverses of them as in Theorem $7.7$ of \cite{SaWe08}, $d\mathcal{F}_{\pi,\Psi^{w-1}}|_{(A^{n_j},0)}$ also has a right inverse for sufficiently large $j$. This is a contradiction and we have the conclusion of Step \ref{k}.

For induction, we need to show:
\begin{step}\label{step4}
There exists the perturbation $\Psi^w$ satisfying the surjectivity of the operator
\[
d\mathcal{F}_{\pi,\Psi^w}:\A^W(a_w)_\delta \times \R^d\ri \Om^+(W)\otimes \su_{L^2_{q-1,\delta}}
\]
for any point in  $(\mathcal{F}_{\pi,\Psi^w})^{-1}(0)\cap (\A(a_w)_{\delta}\times \{ 0\} )$.  
\end{step}
We take the perturbation $\Psi_A=((l^j_A),(\mu^+_j(A)))$ for each $ A\in K_w$ in Step 1. Because $K_w$ is compact and surjectivity of the operators is open condition, there exist $\{A_1,\cdots ,A_k\} \subset K_w$ and a perturbation $\Psi^w$ such that
\[
d\mathcal{F}_{\pi,\Psi^w}|_{(A,0)}:\A^W(a_w)_\delta \times \R^d\ri \Om^+(W)\otimes \su_{L^2_{q-1,\delta}}
\]
is surjective for all $(A,0) \in (\mathcal{F}_{\pi,\Psi^w})^{-1}(0) \cap (\A^W(a)_{\delta}\times \{ 0\})$.
Here $\Psi^w$ is defined by 
\[
\Psi^w:=((l^j_{A_1} \cdots l^j_{A_k}, l^{w-1}_i) ,(\mu^+_j(A_1),\cdots,\mu^+_j(A_k), \mu^{w-1}_i))
\]
which satisfies the property in Step \ref{step4}.

Second, we show that the operator $d\mathcal{F}_{\pi,\Psi^w}$ is surjective for any point in  ${\mathcal{F}_{\pi,\Psi^w}}^{-1}(0)\cap (\A^W(a)_{\delta}\times  D^d(\eta) )$. Suppose there is no $\eta$ such that the statement holds. Then there is a sequence $\{(A_n,\epsilon_n)\}$ in $M^W(a)_{\delta,\pi,\Psi,\epsilon_n}$ which satisfies that $\epsilon_n \ri 0$ as $n\ri \infty$ and $d\mathcal{F}_{\pi,\Psi^w}|_{(A_n,\epsilon_n)}$ is not surjective for all $n\in \n$. Because the bubble does occur, $\{A_n\}$ has a chain convergent subsequence to 
\[
([B^1],\dots , [B^N],[A^0]) \in M(b_0,b_1)_\pi \times \dots \times M(b_{v-1},b_v)_\pi \times M^W(b_v)_{\delta,\pi,\Psi,0}
\]
 for some $b_i \in \widetilde{R}(Y)_\pi$. Since $\pi$ is a regular perturbation and $d\mathcal{F}_{\pi,\Psi^w}|_{(A^0,0)}$ is surjective, there exist the right inverses of $d^+_{B^1}+d\pi^+_{B^1}, \dots,d^+_{B^N}+d\pi^+_{B^N}$ and $d\mathcal{F}_{\pi,\Psi^w}|_{(A^0,0)}$ for suitable functional spaces. By the gluing of the right inverses as in Step \ref{k}, $d\mathcal{F}_{\pi,\Psi^w}|_{(A_N,\epsilon_N)}$ also has the right inverse for large $N$. This is a contradiction and this completes the proof.
\qed
\end{proof}
\begin{thm}\label{Tra} For a given data $(\delta,\pi,a)$ in Lemma \ref{lem:Sur}, there exist $\eta>0$, a perturbation $\Psi$ and a dense subset of $R \subset$ $B^d(\eta) \subset \R^d$ such that $(\Psi,b)$ is a regular perturbation for $b \in R$.
\end{thm}
\begin{proof}
This is a conclusion of Lemma \ref{lem:Sur}, the argument in Section $3$ of \cite{FU91} and the Sard-Smale theorem. 
\qed
\end{proof}

Applying the implicit function theorem, we get a structure of manifold of $M^W(a)_{\delta,\pi,\Psi,b}$. Its dimension coincides with the Floer index $\ind(a)$ of $a$ by Proposition \ref{calfred}. Therefore we have:
\begin{cor}\label{trans}
For given data $(\delta,\pi,a)$ in Lemma \ref{lem:Sur}, there exist $\eta>0$, a perturbation $\Psi$ and a dense subset of $R \subset$ $B^d(\eta) \subset \R^d$ such that $M^W(a)_{\delta,\pi,\Psi,b}$ has a structure of manifold of dimension $\ind(a)$.
\end{cor}

\subsection{Orientation}\label{ori}
In \cite{Do87}, Donaldson showed the orientability of the instanton moduli spaces for closed oriented 4-manifolds. In this subsection, we
deal with the case for non-compact 4-manifold $W$ by generalizing Donaldson's argument.  More explicitly, we show that the moduli space $M^W(a)_{\delta,\pi}$ is orientable. We also follow Fredholm and moduli theory in \cite{T87} to formulate the configuration space for $SU(l)$-connections for $l\geq 2$.

Let $Z$ be a compact oriented 4-manifold which satisfies $\partial Z=Y$ and $H_1(Z)\cong 0$. We set $Z^+:=(-Z)\cup_Y W[0,\infty]$ and $\hat{Z}:= (-Z)\cup_Y Y\times [0,\infty)$. Fix a Riemannian metric $g_{Z^+}$ on $Z^+$ with $g_{Z^+}|_{W[0,\infty]}=g_W|_{W[0,\infty]}$ and Riemannian metric $g_{\hat{Z}}$ with $g_{\hat{Z}}|_{Y\times [0,\infty)}= g_Y \times g^{\text{stan}}_\R$. First, we introduce the configuration spaces for $SU(l)$-connections on $W$ and $Z^+$ for $l\geq 2$ and {\it $SU(2)$-configuration space} for $\hat{Z}$.
\begin{defn}\upshape 
Fix a positive integer $q \geq3$. For an irreducible $SU(2)$-connection $a$ on $Y$, we define
\[
\A^{W}(a)_{(\delta,\delta),l}:= \left\{A_{a}+c \middle | c \in \Om^1(W)\otimes \mathfrak{su}(l)_{L^2_{q,(\delta,\delta)}} \right\},
\]
\[
\A^{Z^+}_{\delta,l}:= \left\{\theta+c \middle | c \in \Om^1(Z^+)\otimes \mathfrak{su}(l)_{L^2_{q,\delta}} \right\},
\]
and
\[
\A^{\hat{Z}}(a):= \left\{B_a+c \middle | c \in \Om^1(\hat{Z})\otimes \mathfrak{su}(2)_{L^2_{q}} \right\},
\]
where 
\begin{itemize}
\item $A_{a}$ is an $SU(l)$-connection on $W$ with $A_{a}|_{Y\times (-\infty,-1]}=\pr^* (a\oplus \theta)$, $A_{a}|_{W[0,\infty] }=\theta$ 
\item $B_a$ is an $SU(2)$-connection on $\hat{Z}$ with $B_{a}|_{Y\times [0,\infty)}=\pr^*a$.
\item
$L^2_{q,(\delta,\delta)}(W)$-norm is defined by
\[
||f||^2_{L^2_{q,(\delta,\delta)}(W)}:= \sum_{0\leq i \leq q} \int_W e^{\tau' \delta}  |\nabla_{A_a}^i f |^2d\text{vol} ,
\]
where $\tau'$ is defined in Definition \ref{confset} and $f$ is an element in $\Om^1(W)\otimes \su$ with compact support.
\item $L^2_{q,\delta}(Z^+)$-norm is defined by
\[
||f||^2_{L^2_{q,\delta}(Z^+)}:= \sum_{0\leq i \leq q} \int_{Z^+} e^{\tau'' \delta}  |\nabla^i_\theta f |^2d\text{vol},
\]
where $\tau'':Z^+\ri [0,1]$ is a smooth function satisfying $\tau''|_{W[0,\infty]}=\tau$ defined in Definition \ref{confset} and  $f$ is an element in $\Om^1(Z^+)\otimes \su$ with compact support.
\item $L^2_{q}(\hat{Z})$-norm is defined by
\[
||f||^2_{L^2_{q}(\hat{Z})}:= \sum_{0\leq i \leq q} \int_{\hat{Z}}  |\nabla_{B_a}^i f |^2d\text{vol} ,
\]
where  $f$ is an element in $\Om^1(\hat{Z})\otimes \su$ with compact support.
\end{itemize}
We also define {\it the $SU(l)$ configuration spaces} $\B^W (a)_{\delta,\delta}^{l}$, $\B^{Z^+}_{\delta,l}$ and {\it $SU(2)$-configuration spaces} $\B^{\hat{Z}}(a)$ by
\[
\B^W (a)_{(\delta,\delta),l}:=\A^W(a)_{(\delta,\delta),l} /\G^W(a)_{l},\ \B^{Z^+}_{\delta,l}:=\A^{Z^+}_{\delta,l} /\G^{Z^+}_{l}
\]
and
\[
\B^{\hat{Z}}(a):=\A^{\hat{Z}}(a) /\G^{\hat{Z}}(a)
\]
where $\G^W(a)_{l}$, $\G^{Z^+}_{l}$ and $\G^{\hat{Z}}(a)$ are given by 
\[
\G^W(a)_l:=\left\{ g\in \aut(W\times SU(l)) \subset \End(\mathbb{C}^l)_{L^2_{q+1,\text{loc}}} \middle| \nabla_{A_a}(g) \in L^2_{q,(\delta,\delta)}(W)\right\} ,
\]
\[
\G^{Z^+}_l:=\left\{ g\in \aut(Z^+\times SU(l)) \subset \End(\mathbb{C}^l)_{L^2_{q+1,\text{loc}}} \middle| d(g) \in L^2_{q,\delta}(Z^+) \right\}
\]
and
\[
\G^{\hat{Z}}(a):= \left\{ g\in \aut(\hat{Z}\times SU(l)) \subset \End(\mathbb{C}^2)_{L^2_{q+1,\text{loc}}} \middle| \nabla_{B_a}(g) \in L^2_{q,\delta}(\hat{Z}) \right\}.
\] The action of $\G^W(a)_l$(resp. $\G^{Z^+}_l$, $\G^{\hat{Z}}(a)$) on $\A^{W}(a)_{(\delta,\delta).l}$(resp. $\A^{Z^+}_{\delta,l}$, $\A^{\hat{Z}}(a)$)  is the pull-backs of connections.
We define {\it the reduced gauge group} by 
\[
\hat{\G}^{W}(a)_l:= \left\{ g \in \G^W(a)_l \middle |\lim_{t \ri -\infty} g|_{Y\times t} =id\  \right\} ,
\]
\[
\hat{\G}^{W,\text{fr}}(a)_l:=\left\{ g \in \hat{\G}^{W}(a)_l \middle| \lim_{n \ri \infty} g|_{W_n} \ri id \right\}
\]
and
\[
\hat{\G}^{Z^+}_l:= \left\{ g \in \G^{Z^+}_l \middle | \lim_{n \ri \infty} g|_{W_n} \ri id  \right\}.
\]
Then we define
\[
\hat{\B}^{W} (a)_{(\delta,\delta),l} :=\A(a)^{W}_{(\delta,\delta)}/ \hat{\G}^{W}(a)_l,\ \hat{\B}^{W,\text{fr}} (a)_{(\delta,\delta),l} :=\A(a)^{W}_{(\delta,\delta)}/ \hat{\G}^{W,\text{fr}}(a)_l
\]
and
\[
 \hat{\B}^{Z^+}_{\delta,l}:=\A^{Z^+}_{\delta,l}/ \hat{\G}^{Z^+}_{l}.
\]
\end{defn}
The group $\hat{\G}^{W,\text{fr}}(a)_l$ (resp. $\hat{\G}^{Z^+}_l$) has a structure of Banach Lie sub group of $\G^W(a)_l$ (resp. $\G^{Z^+}_l$). By the construction of them, there are exact sequences 
\[
\hat{\G}^{W}(a)_l \ri \G^W(a)_l \ri \stab(a\oplus \theta),\ \hat{\G}^{W,\text{fr}}(a)_l\ri \hat{\G}^W(a)_l \ri SU(l)
\]
and
\[
\hat{\G}^{Z^+}_l\ri {\G}^{Z^+}_l \ri SU(l)
\]
of Lie groups.
The group $\hat{\G}^{W}(a)_l$(resp. $\hat{\G}^{Z^+}_l$) acts on $\A^{W} (a)_{(\delta,\delta),l}$(resp. $\A^{Z^+}_{\delta,l})$ freely.

\begin{prop}\label{simp}
For $l\geq 3$ and an $SU(2)$-flat connection $a$, there exists a positive number $\delta'$ such that for a positive real number $\delta$ less than $\delta'$ the following properties hold. 
\begin{itemize}
\item $\hat{\B}^{W} (a)_{(\delta,\delta),l}$ is simply connected.
\item  $\hat{\B}^{Z^+}_{\delta,l}$ is simply connected.
\end{itemize}
\end{prop}
\begin{proof} We will show only the first property. The second one is shown in a similar way to the first case. We use the condition $H_1(Z,\z)\cong 0$ for the second property.

 Since $\pi_i(SU(l))=0$ for $i=0,1$, 
 \[
 \pi_1(\hat{\B}^{W}(a)_{(\delta,\delta),l} )\text{ is isomorphic to }\pi_1(\hat{\B}^{W,\text{fr}} (a)_{(\delta,\delta),l} ).
 \]
Therefore, we will show $\pi_1(\hat{\B}^{W,\text{fr}} (a)_{(\delta,\delta),l})=0$.
There exists $\delta'>0$ such that for $0< \delta<\delta'$, 
\begin{align}\label{fib}
\hat{\G}^{l}(a)_{(\delta,\delta)}^f \ri \A^{W} (a)_{(\delta,\delta),l} \ri \hat{\B}^{W,\text{fr}} (a)_{(\delta,\delta),l}
\end{align}
is a fibration since \eqref{fib} has a local slice due to Fredholm and moduli theory in \cite{T87}. Let $W^*$ be the one point compactification of $W$. Using $\eqref{fib}$, we obtain 
\[
\pi_1(\hat{\B}^{W,\text{fr}} (a)_{(\delta,\delta),l}) \cong \pi_0(\hat{\G}^{l}(a)_{(\delta,\delta)}^f ) \cong [W^*,SU(l)].
\]
Since $\pi_i(SU(l))$ vanishes for $i=0,1,2,4$, the obstruction for an element of $ [W^*,SU(l)]$ to be homotopic to the constant map lives in $H^3(W^*,\pi_3(SU(l))) \cong H^3_{\text{comp}}(W,\pi_3(SU(l))) \cong H_1(W,\pi_3(SU(l)))=0$ where the second isomorphism is the Poincar\'e duality.

This implies 
\[
\pi_1(\hat{\B}^{W,\text{fr}} (a)_{(\delta,\delta),l})\cong 0.
\]
\qed
\end{proof}

We now define the determinant line bundles. 
For simplify, we impose Assumption $\ref{imp}$ on $Y$.
\begin{defn}\upshape
Let $\pi$ be an element in $\prod(Y)^{\text{flat}}$ and $(\Psi,\epsilon)$ be a perturbation in Subsection \ref{hol2} and fix an element $a \in \widetilde{R}(Y)$. 
For $c \in \B^{W}(a)_{(\delta,\delta),l}$ ($\hat{\B}^{W}(a)_{(\delta,\delta),l})$, we have the following bounded operator 
\[
d(\mathcal{F}_{\pi,\Psi})_c+d^{*_{L^2_{(\delta,\delta)}}}_c:\ome^1(W) \otimes \fr{su}(l)_{L^2_{q,(\delta,\delta)}} \ri \ome^0(W) \otimes \fr{su}(l) \oplus \ome^+(W) \otimes \fr{su}(l))_{L^2_{q-1,(\delta,\delta)}}.
\]
The operators $d(\mathcal{F}_{\pi,\Psi})_c+d^{*_{L^2_{(\delta,\delta)}}}_c$ are the Fredholm operators for small $\delta$. Fix such a $\delta$.
We set
\[
\lambda^W (a,l,c):= \Lambda^{\max} \ker (d(\mathcal{F}_{\pi,\Psi})_c) \otimes  \Lambda^{\max} \coker (d(\mathcal{F}_{\pi,\Psi})_c)^*.
\]
{\it The determinant line bundles} are defined by
\[
\lambda^W (a,l):=\displaystyle \bigcup_{c \in \B^W(a)_{(\delta,\delta),l}} \lambda (a,l,c) \ri \B^W(a)_{(\delta,\delta),l}
\]
and
 \[
 \hat{\lambda}^{W} (a,l):=\displaystyle \bigcup_{c \in \hat{\B}^{W}(a)_{(\delta,\delta),l}} \lambda (a,l,c) \ri \hat{\B}^{W}(a)_{(\delta,\delta),l}.
\]
We also define 
\[
\lambda^{Z^+} (l)\ri {\B}^{Z^+}_{\delta,l} \text{ and } \lambda^{\hat{Z}}(a) \ri {\B}^{\hat{Z}}(a) 
\]
in a similar way with respect to the operators
\[
d^+_c+d^{*_{L^2_{\delta}}}_c :\ome^1(Z^+) \otimes \fr{su}(l)_{L^2_{q,\delta}} \ri \ome^0(Z^+) \otimes \fr{su}(l) \oplus \ome^+(Z^+) \otimes \fr{su}(l))_{L^2_{q-1,\delta}} 
\]
for $c \in {\B}^{Z^+}_{\delta,l}$ and
\[
d^+_c+d^*_c :\ome^1(\hat{Z}) \otimes \su_{L^2_{q}} \ri\ome^0(\hat{Z}) \otimes \su \oplus \ome^+(\hat{Z}) \otimes \su)_{L^2_{q-1}}
\]
for $c \in {\B}^{\hat{Z}}(a)$.

\end{defn}
\begin{lem}For a given data $(a,\delta,l)$ in Proposition \ref{simp}, 
the bundles $\lambda^{Z^+} (a,l)\ri {\B}^{Z^+}_{\delta,l}$ and $\lambda^W(a,l)\ri \B^W(a)_{(\delta,\delta),l}$ are trivial.
\end{lem}
\begin{proof}Since the determinant line bundle is a real line bundle, the triviality of $\lambda^{Z^+} (a,l)\ri {\B}^{Z^+}_{\delta,l}$ is a consequence of Proposition $\ref{simp}$. Therefore we show the triviality of  $\lambda^W(a,l)\ri \B^W(a)_{(\delta,\delta),l}$.
We have a fibration 
\begin{align}\label{fib1}
\stab(a\oplus \theta) \ri \hat{\B}^W(a)_{(\delta,\delta),l} \xri{j} \B^W(a)_{(\delta,\delta),l}.
\end{align}
We also have an isomorphism $j^*\lambda^W(a,l) \cong \hat{\lambda}^W(a,l)$ for $j$ in \eqref{fib1}. 
$\hat{\lambda}^W(a,l)$ is the trivial bundle for $l>2$ from Proposition $\ref{simp}$. So if the fiber $\stab(a\oplus \theta)$ of \eqref{fib1} is connected, $\lambda^W(a,l)$ is also trivial. 
The possibilities of $\stab(a\oplus \theta)$ are $SU(l)$, $U(1)\times U(l-1)$, $S(U(2)\times U(l-2))$ and 
\[
\left\{(z,A) \in U(1)\times U(l-2) \middle| z^2 \text{det}A=1\right\}.
\]
Since these groups are connected, $\lambda^W(a,l)$ is the trivial bundle.
\qed
\end{proof}
\begin{lem}\label{222}Suppose that $Y$ satisfies Assumption \ref{imp} and $a$ is an element in $\widetilde{R}^*(Y)$.
Let $i_1:\B^W(a)_{(\delta,\delta),2} \ri \B^W(a)_{(\delta,\delta),3}$ and $i_2:\B^{Z^+}_{\delta,2} \ri \B^{Z^+}_{\delta,3}$ be the maps induced by the product with the product connection. 
 There exists a positive number $\delta'$ such that for a positive real number $\delta$ less than $\delta'$, $i_1^*\lambda^W(a,3) \cong \lambda^W(a,2)$ and $i_2^*\lambda^{Z^+}(a,3) \cong \lambda^{Z^+}(a,2)$ hold.
\end{lem}
\begin{proof}
Under  Assumption \ref{imp} on $Y$, the isomorphism class of these line bundles are independent of the choices of the perturbations $\pi$ and $(\Psi,\epsilon)$ by considering a 1-parameter family of perturbations $\pi_t:=(f,th)$ and $(\Psi,t\epsilon)$ for $t \in [0,1]$.
So Lemma $(5.4.4)$ in \cite{DK90} implies the conclusion.
\end{proof}
\begin{thm}\label{orie}Suppose that $Y$ satisfies Assumption \ref{imp} and $a$ is an element in $\widetilde{R}^*(Y)$. Let $\pi $ be an element in $\prod(Y)^{\text{flat}}$ and $(\Psi, \epsilon)$ be a regular perturbation for $a \in \widetilde{R}^*(Y)$. For sufficiently small $\delta$, $M^W(a)_{\pi,\Psi,\epsilon,\delta}$ is orientable. Furthermore the orientation of $M^W(a)_{\pi,\Psi,\epsilon,\delta}$ is induced by the orientation of $\lambda^W(a,2)$.
\end{thm}
\begin{proof}
Using the exponential decay estimate in Proposition 4.3 of \cite{Do02}, we have a inclusion $i:M^W(a)_{\pi,\Psi,\epsilon,\delta} \ri B^W(a)_{(\delta,\delta),2}$ for small $\delta$ as a set. From this inclusion $i$, we regard $M^W(a)_{\pi,\Psi,\epsilon,\delta}$ as a subset in $B^W(a)_{(\delta,\delta),2}$. Applying result of convergence in Corollary $5.2$ of \cite{Do02}, we can show that the topology of  $M^W(a)_{\pi,\Psi,\epsilon,\delta}$ in $B^W(a)_{\delta}$ coincides with the topology of $M^W(a)_{\pi,\Psi,\epsilon,\delta}$ in $B^W(a)_{(\delta,\delta),2}$
Also using exponential decay estimate for solutions to the linearized equation in Lemma $3.3$ of \cite{Do02} , $\lambda^W(a,2)|_{M^W(a)_{\pi,\Psi,\epsilon,\delta}} \ri M^W(a)_{\pi,\Psi,\epsilon,\delta}$ is canonically isomorphic to $\Lambda^{\max} M^W(a)_{\pi,\Psi,\epsilon,\delta}$. \qed
\end{proof}

From Theorem \ref{orie}, an orientation of $M^W(a)_{\pi,\Psi,\epsilon,\delta}$ is characterized by the trivialization of $\lambda^W(a,2)$. On the other hand, to formulate the instanton Floer homology of $Y$ with $\z$ coefficient, Donaldson introduced the line bundle $\lambda(a):=\lambda^{(-\hat{Z})}(a) \otimes \lambda_{(-\hat{Z})}^* \ri \B^{(-\hat{Z})}(a)$, where $\lambda_{(-\hat{Z})}$ is given by
\[
\Lambda^{\max} (H^0_{DR}(-\hat{Z})\oplus H^1_{DR}(-\hat{Z}) \oplus H^+_{DR}(-\hat{Z}))
\]
in Subsection $5.4$ of \cite{Do02}.
The orientation of $\lambda(a)$ is essentially independent of the choice of $Z$.

\begin{defn}\upshape
We set
\[
\lambda_W:= \Lambda^{\max} (H^0_{DR}(W)\oplus H^1_{DR}(W) \oplus H^+_{DR}(W)),
\]
and
$\lambda^W(a):= \lambda^W(a,2) \otimes \lambda_W \ri \B^{W} (a)_{(\delta,\delta),2}$.
\end{defn}
\begin{lem}\label{orien}
Suppose that $Y$ satisfies Assumption \ref{imp}. For an irreducible flat connection $a$, there is a canonical identification between the orientations of $\lambda^W(a )$ and the orientations of $\lambda(a)$.
\end{lem}
\begin{proof}
It suffices to construct an isomorphism $\lambda^W(a)\cong \lambda(a)$ which is canonical up to homotopy.
First we fix two elements $[A] \in \B^{W} (a)_{(\delta,\delta)}$ and $[B] \in \B^{\hat{Z}}(a)$ which have representative $A$ and $B$ satisfying $A_{Y\times (-\infty,-1]}=\pr^*a$ and $B|_{Y\times [1,\infty)}=\pr^*a$.
For such two connections, we obtain an element $A\# B \in \B^{Z^+}_\delta$ by gluing of connections.
This map induces an isomorphism
\[ 
\#:det( d_A^*+d^+_A)\otimes det( d_B^*+d^+_B)\ri det( d_{A\# B}^*+d^+_{A\# B})
\]
from the similar argument of Proposition $3.9$ in \cite{Do02}.
Therefore we have identification
\[
\#:\lambda^W(a)|_{[A]}\otimes \lambda^{Z^+}(a)|_{[B]} \ri \lambda_{Z^+}|_{[A\#B]}.
\]
If we choose a path from $[\theta]$ to $[A\#B]$ in $\B^{Z^+}_\delta$, then we have an identification between $\lambda_{Z^+}|_{[\theta]}$ and $\lambda_{Z^+}|_{[A\#B]}$. The line bundle $\lambda_{Z^+}|_{[\theta]}$ is naturally isomorphic to 
\[
\Lambda^{\max} (H^0_{DR}(Z^+)\oplus H^1_{DR}(Z^+) \oplus H^+_{DR}(Z^+))
\]
 by using Proposition $5.1$ in \cite{T87}. This cohomology group is isomorphic to
\[
\Lambda^{\max} (H^0_{DR}(Z)\oplus H^1_{DR}(Z) \oplus H^+_{DR}(Z))\otimes \Lambda^{\max} (H^0_{DR}(W)\oplus H^1_{DR}(W) \oplus H^+_{DR}(W))
\]
by using the Mayer-Vietoris sequence.

Therefore
\[
\lambda^{W}(a)|_{[A]}\otimes \lambda_W^* \cong (\lambda^{(-Z)}(a)|_{[B]} \otimes \lambda_{(-Z)}^*)^*
\]
holds. Because $\lambda^{Z^+}(a,2)\ri \B^{Z^+}_{\delta,2}$ is orientable by Lemma \ref{222} and Proposition \ref{simp} , the homotopy class of this identification does not depend on choices of the path, $A$, $B$, the bump functions of the gluing map.
We also have the following canonical isomorphism 
\[
(\lambda^{(-Z)}(a)|_{[B]} \otimes \lambda_{(-Z)}^*)^* \cong \lambda^Z(a)|_{[B]} \otimes \lambda_{Z}^*.
\]
by the gluing $Z$ and $-Z$ as above discussion and the Mayer-Vietoris sequence.
This completes the proof.
\qed
\end{proof}
Combining Theorem \ref{orie} and Lemma \ref{orien}, we have:
\begin{thm}\label{ori:conc}Under the assumption of Theorem \ref{orie}, an orientation $\lambda(a)$ and an orientation of $\lambda_W$ give an orientation of $M^W(a)_{\pi,\Psi,\epsilon,\delta}$.
\end{thm}

\section{Proof of main theorem}\label{promain}

\begin{proof}
Let $Y$, $l_Y$ and $X$ be as in Section $\ref{main}$. 
Take a Riemannian metric $g_Y$ on $Y$. Fix a non-negative real number $r \in \Lambda_Y$ smaller than $Q^{2l_Y+3}_X$. Suppose that there is an embedding $f$ of $Y$ into $X$ satisfying $f_*[Y]= 1 \in H_3(X,\z)$. Then we obtain the oriented homology cobordism from $Y$ to $-Y$ by cutting open along $Y$. Recall that $W$ is a non-compact oriented Riemann 4-manifold $W$ with both of cylindrical end and periodic end which is formulated at the beginning of Subsection \ref{fred}.

We fix a holonomy perturbation $\pi \in \prod(Y)$ satisfying the following conditions.
\begin{enumerate}
\item $\pi$ is a $\epsilon$-perturbation in Subsection \ref{filter}.
\item $\pi$ is a regular perturbation in the end of Subsection \ref{hol}.
\item $\pi$ is an element of $\prod(Y)^{\text{flat}}$ in Definition $\ref{flatpres}$.  
\item For $a\in \widetilde{R}(Y)$ with $0\leq cs_Y(a) < \min\{1,Q^{2l_Y+3}_X\}$ and $A \in M^W(a)_{\pi,\delta}$,
\[
\frac{1}{8\pi^2}\sup_{n \in \n} ||F(A)+s\pi(A)||^2_{L^2(W)} <\min\{1,Q^{2l_Y+3}_X\}
\]
holds.
\item $d^+_\theta+  sd\pi_\theta:\A^W(\theta)_{(\delta,\delta)}\times \R^d \ri \Omega^+(W)\otimes \su_{L^2_{q-1,(\delta,\delta)}}$ is surjective.
\item For $c \in \widetilde{R}(Y)_\pi$ satisfying $cs_{Y}(c)<0$, $M^W(c)_{\pi,\delta}$ is the empty set.
\end{enumerate}
Assumption \ref{imp} and the proof of Thereom $8.4$ (ii) of \cite{SaWe08} implies the existence of the perturbation ssatisfying the third condition. The first, forth, fifth and sixth conditions follow from choosing small $h\in C^{l'}(SU(2)^{d'},\R)_{ad} $ of $\pi=(f,h)$.  

Next we also fix a holomomy perturbation $(\Psi,\epsilon)$ satisfying the following conditions.
\begin{enumerate}
\item $(\Psi,\epsilon)$ is a regular for $[b] \in \widetilde{R}(Y)$ with $0\leq cs_Y(b) \leq cs_Y(a)$. 
\item For $a\in \widetilde{R}(Y)$ with $0\leq cs_Y(a) < \min\{1,Q^{2l_Y+3}_X\}$,
\[
\frac{1}{8\pi^2}\sup_{n \in \n} ||F(A)+s\pi(A)+\sigma_{\Psi}(A,\epsilon)||^2_{L^2(W)} <\min\{1,Q^{2l_Y+3}_X\}
\]
hold.
\end{enumerate}
To get the first condition, we use Lemma $\ref{lem:Sur}$. The second condition satisfied when we take $\epsilon$ sufficiently small.

 In order to formulate the instanton Floer homology of $Y$ with $\z$ coefficient, we fix an orientation of fix an orientation of $\lambda(a)$ for each $a \in R(Y)$. The orientation of $Y$ induce an orientation of $\lambda_W$. To determine the orientation of $M^W(a)_{\pi,\Psi,\epsilon,\delta}$, we fix a compact oriented manifold $Z$ with $H_1(Z,\z)\cong 0$ as in Subsection $\ref{ori}$.
 The relation between $\lambda_{Z,a}$ and $\lambda(a)$ is given by
 \[
 \lambda^Z(a) \otimes \lambda_{Z} \cong \lambda(a).
 \]

 Let $a$ be a flat connection satisfying $cs_Y(a)<r\leq \min \{Q_X^{2l_Y+3},1\}$ and $\ind(a)=1$. We consider the moduli space $M^W(a)_{\pi,\Psi,\epsilon,\delta}$. From the choice of these perturbation data and Corollary \ref{trans}, $M^W(a)_{\pi,\Psi,\epsilon,\delta}$ has a structure of 1-dimensional manifold for small $\delta$. From Theorem \ref{ori:conc}, we obtain an orientation of  $M^W(a)_{\pi,\Psi,\epsilon,\delta}$ induced by the orientation of $\lambda_{Z,a}$.  
 
 Let $(A,B)$ be a limit point of $M^W(a)_{\pi,\Psi,\epsilon,\delta}$. Using Theorem $\ref{cptness}$ and the standard dimension counting argument, the limit points of $M^W(a)_{\pi,\Psi,\epsilon,\delta}$ correspond to two cases:
\begin{enumerate}
\item $\displaystyle(A,B) \in \bigcup_{b \in \widetilde{R}^*(Y), cs_Y(b)<r,\ind(b)=0} M(a,b)_\pi \times M^W(b)_{\pi,\Psi,\epsilon,\delta}$
\item $(A,B) \in M(a,\theta)_\pi \times M^W(\theta)_{\pi,\Psi,\epsilon,(\delta,\delta)}$.
\end{enumerate}
For the second case, we use the exponential decay estimate to show $B \in M^W(\theta)_{\pi,\Psi,\epsilon,(\delta,\delta)}$.
Here $M(a,b)_\pi$ and $M(a,\theta)_{\pi,\delta}$ have a structure of 1-dimensional manifold. The quotient spaces $M(a,b)_\pi/\R$ and $M(a,\theta)_{\pi,\delta}/\R$ have a structure of compact oriented $0$-dimensional manifold whose orientation induced by the orientation of $\lambda_Z$ and $\R$ action by the translation as in Subsection $5.4$ of \cite{Do02}. Corollary \ref{trans} and Theorem \ref{ori:conc} imply that $M(b)_{\pi,\Psi,\epsilon,\delta}$ has a structure of compact oriented 0-manifold whose orientation induced by the orientation of $\lambda_b$ and $\lambda_W$ for small $\delta$. Since the formal dimension of $M(\theta)_{\pi,\Psi,\epsilon,(\delta,\delta)}$ is $-3$ from Proposition \ref{calfred} and there is no reducible solution except $\theta$ for a regular perutrbation $(\Psi,\epsilon)$, $M(\theta)_{\pi,\Psi,\epsilon,\delta}$ consists of just one point. By the gluing theory as in Theorem $4.17$ and Subsection $4.4.1$ of \cite{Do02}, there is the following diffeomorphism onto its image:
\[
\mathcal{J}:\left(\bigcup_{b \in \widetilde{R}^*(Y), cs_Y(b)<r, \ind(b)=0}(M(a,b)_\pi/\R \times M(b)_{\pi,\Psi,\epsilon,\delta}) \cup M(a,\theta)_\pi/\R \right)\times [T,\infty)
\]
\[
 \ri M^W(a)_{\pi,\Psi,\epsilon,\delta}.
\]
By the definition of the orientation of $M(a,b)_\pi/\R$ and $M(a,\theta)_\pi/\R$, we can construct $\mathcal{J}$ as an orientation preserving map. Furthermore, the complement of $\im \mathcal{J}$ is compact. Therefore we can construct the compactification of $M^W(a)_{\pi,\Psi,\epsilon,\delta}$ by adding the finitely many points 
\begin{align}\label{main'}
\bigcup_{b \in \widetilde{R}^*(Y), cs_Y(b)<r,\ind(b)=0}  (M(a,b)_\pi/\R\times M(b)_{\pi,\Psi,\epsilon,\delta}) \cup  M(a,\theta)_\pi/\R,
\end{align}
which has a structure of compact oriented  $1$-manifold.
By counting of boundary points of the compactification, we obtain the relation 
\[
\delta^{r}( n)(a)+ \theta^{r}(a)=0,
\]
where $n \in CF^0_r(Y)$ is defined by $n(b):= \# M(b)_{\pi,\Psi,\epsilon,\delta}$. This implies $\theta^{r}$ is a coboundary. Therefore we have $0=[\theta^r] \in HF^1_{r}(Y)$ for $0\leq r \leq \min \{Q^{2l_Y+3}_X,1\}$.
\qed
\end{proof}
\bibliographystyle{jplain}
\bibliography{Instantons}
\end{document}